\documentclass[10pt]{amsart}

\usepackage{graphicx}
\usepackage{amsmath, amssymb,amsthm}
\usepackage{pdflscape}
\usepackage{longtable}
\usepackage{float}
\usepackage{tikz}
\usepackage{tikz-3dplot} 
\usepackage{hyperref}
\usepackage{subfigure}

\usepackage{mathtools}

\theoremstyle{plain}
\newtheorem{thm}{Theorem}[section]

\newtheorem{cor}[thm]{Corollary}
\newtheorem{lem}[thm]{Lemma} 
\newtheorem{prop}[thm]{Proposition}

\theoremstyle{definition}
\newtheorem{defi}[thm]{Definition}
\newtheorem{defit}[thm]{Theorem-Definition}
\theoremstyle{remark}
\newtheorem{rem}[thm]{Remark}
\newtheorem{ex}[thm]{Example}

\newcommand{\lgw}{\longrightarrow}
\newcommand{\lgm}{\longmapsto}

\newcommand{\ovl}{\overline}

\newcommand{\ord}{\text{ord}}

\newcommand{\Int}{\operatorname{Int}}

\newcommand{\x}{\mathbf{x}}

\newcommand{\wdt}{\widetilde}
\newcommand{\Supp}{\text{Supp}}
\newcommand{\D}{\Delta}
\newcommand{\G}{\Gamma}
\newcommand{\la}{\lambda}

\renewcommand{\L}{\mathbb{L}}

\newcommand{\Z}{\mathbb{Z}}

\newcommand{\GG}{G(p)}
\renewcommand{\k}{\Bbbk}

\newcommand{\R}{\mathbb{R}}
\newcommand{\K}{\mathbb{K}}

\renewcommand{\SS}{\mathcal{S}}

\newcommand{\N}{\mathbb{N}}
\newcommand{\C}{\mathbb{C}}

\newcommand{\Q}{\mathbb{Q}}

\renewcommand{\t}{\tau}
\newcommand{\p}{\textsl{p}}
\newcommand{\Conv}{\operatorname{Conv}}
\renewcommand{\lg}{\langle}
\newcommand{\s}{\sigma}
\newcommand{\rg}{\rangle}

\renewcommand{\a}{\alpha}

\newcommand{\g}{\gamma}

\renewcommand{\phi}{\varphi}

\newcommand{\ZR}{\operatorname{ZR}}
\newcommand{\Ord}{\operatorname{Ord}}
\renewcommand{\o}{\omega}

\newcommand{\NP}{\operatorname{NP}}

\begin{document}
\baselineskip=14pt

\title{The minimal cone of an algebraic Laurent series}
\author{Fuensanta Aroca}
\email{fuen@im.unam.mx}
\address{Instituto de Matem\'aticas, Universidad Nacional Aut\'onoma de M\'exico (UNAM), Mexico}

\author{Julie Decaup}
\email{julie.decaup@im.unam.mx}
\address{Instituto de Matem\'aticas, Universidad Nacional Aut\'onoma de M\'exico (UNAM), Mexico}

\author{Guillaume Rond}
\email{guillaume.rond@univ-amu.fr}
\address{Universit\'e Publique, CNRS, Centrale Marseille, 
I2M, UMR 7373, 13453 Marseille,
France}


\keywords{power series rings, support of a Laurent series, algebraic closure, orders on a lattice,  Henselian valued fields}
\thanks{This work has been partially supported by ECOS Project M14M03, and by PAPIIT  IN108216 and IN108320. The third author is deeply grateful to the UMI LASOL of  CNRS where this project has been carried out.}

\begin{abstract}
We study the algebraic closure of $\K(\!(x)\!)$, the field of power series in several indeterminates over a field $\K$. In characteristic zero we show that the elements algebraic over $\K(\!(x)\!)$ can be expressed as  Puiseux series such that the convex hull of its support is essentially a polyhedral rational cone, strengthening the known results.  
In positive characteristic we construct algebraic closed fields containing the field of power series and we give examples showing that the results proved in characteristic zero are longer valid in positive characteristic.
\end{abstract}

\maketitle



\section{Introduction}
When $\K$ is a field   and $x=(x_1,\ldots, x_n)$ is a vector of $n$ indeterminates, we denote by $\K(\!(x)\!)$ the field of formal power series in $n$ indeterminates. The problem we are studying here concerns the determination of an algebraic closure of $\K(\!(x)\!)$ when $\K$ is an algebraically closed field of any characteristic. \\
Let us begin with the characteristic zero case. When $n=1$, the Newton-Puiseux Theorem asserts that the elements that are algebraic over $\K(\!(x)\!)$ are the Puiseux series, i.e. the formal sums of the form $\sum_{k=k_0}^\infty a_kx^{k/q}$ for some positive integer $q$ (cf. \cite{Pu1} and \cite{Pu2}).\\
When $n\geq 2$ there is no known description of the algebraic closure of $\K(\!(x)\!)$. The Abhyankar-Jung Theorem asserts that the roots of a monic polynomial with coefficients in $\K[[ x]]$ whose discriminant is a monomial times a unit are Puiseux series (cf. \cite{J}, \cite{Ab0}, \cite{KV} or \cite{PR}).  But, in general, polynomials with coefficients in $\K[[ x]]$ may not have Puiseux series as roots, as the polynomial $T^2-(x_1+x_2)$. Nevertheless, a result of MacDonald asserts that we may express the elements algebraic over $\K(\!(x)\!)$ as Laurent Puiseux series \cite{Mc}. In order to explain this result let us introduce some terminology. \\
A \emph{(generalized) series} $\xi$ (with support in $\Q^n$ and coefficients in a field $\K$) is a formal sum $\xi=\sum_{\a\in\Q^n}\xi_\a x^\a$, where $x^\a:=x_1^{\a_1}\cdots x_n^{\a_n}$, and the $\xi_\a\in\K$. Its \emph{support}  is the set
$$\Supp(\xi):=\{\a\in\Q^n\mid \xi_\a\neq 0\}.$$
Such a series is called a \emph{Laurent series}  (resp. \emph{Laurent Puiseux series}) if $\Supp(\xi)\subset \Z^n$ (resp. $\Supp(\xi)\subset \frac{1}{k}\Z^n$ for some $k\in\N^*$). \\
The set of generalized series is a commutative group as we can define the sum of two power series in the usual way. But in general the product of two such series is not well defined. To insure the existence of the product of two generalized series, one has to impose that their support is well-ordered for a total order on $\Q^n$ (see \cite{Ri} for example). This is the case for example when we consider Laurent  series whose supports are included in the translation of a given common strongly convex cone (for example see \cite{AK} or \cite[Lemma 3.8]{AR}). Here, a \emph{strongly convex cone} is a cone that does not contain non-trivial linear subspaces. In particular, for a series $\xi$ whose support is included in a strongly convex cone containing ${\R_{\geq0}}^n$, and for $P(x,T)\in\K[[x]][T]$, $P(x,\xi)$ is well defined.\\
We also recall that a \emph{rational cone} is a finitely generated submonoid of $\R^n$ that is generated by vectors with integer coordinates. Then, MacDonald's Theorem (cf \cite[Theorem 3.6]{Mc} - see also \cite{A-I}) asserts that the elements  that are algebraic over $\K(\!(x)\!)$ can be expressed as Puiseux series with support in the translation of a strongly convex rational cone $\s$. Moreover MacDonald showed that, for any given $\o\in{\R_{>0}}^n$ whose coordinates are $\Q$-linearly independent,  $\s$ can be chosen in such a way  that 
\begin{equation}\label{eq_completion}\forall s\in\s\setminus\{\underline 0\},\  \ s\cdot\o>0.\end{equation}
Let us remark that, for $q\in\N^*$, a Laurent series $\xi(x_1,\ldots, x_n)$ is algebraic over $\K(\!(x)\!)$ if and only if $\xi(x_1^{1/q},\ldots, x_n^{1/q})$ is algebraic over $\K(\!(x)\!)$. Therefore, in order to determine an algebraic closure of $\K(\!(x)\!)$ one only needs to determine which are the Laurent series $\xi$ whose support is included in the translation of a rational strongly convex cone $\s$ that are algebraic over $\K(\!(x)\!)$. And by the result of MacDonald, if we fix  $\o\in{\R_{>0}}^n$ whose coordinates are $\Q$-linearly independent, we may even assume that $\s$ satisfies \eqref{eq_completion}. \\
For such a $\o$ we define the monomial valuation $\nu_\o$ in the following way: for $f=\sum_{\a\in\N^n}f_\a x^\a$, we set $\nu_\o(f):=\min\{\a\cdot\o\mid f_\a\neq 0\}$. This valuation defines a norm $\|\cdot\|_\o$ on $\K(\!(x)\!)$ by
$$\|f/g\|_\o:=e^{-\nu_\o(f)+\nu_\o(g)}.$$
We denote by $\L^\o$ the completion of $\K(\!(x)\!)$ with respect to $\|\cdot\|_\o$. Then, we remark that a Laurent series whose support is included in the translation of a  cone $\s$ satisfying \eqref{eq_completion}, is necessarily in $\L^\o$. Therefore  in order to determine an algebraic closure of $\K(\!(x)\!)$ one only needs to determine  the algebraic closure of $\K(\!(x)\!)$ in $\L^\o$, its completion for the norm $\|\cdot\|_\o$. 
Passing through the completion of a field $\k$ in order to understand its algebraic closure is a classical process that appears at least in two important situations:

\begin{enumerate}
\item When we want to understand the algebraic closure of $\Q$, we equip $\Q$ with the usual absolute value, and study the algebraic elements of $\R$, its completion, over $\Q$. Indeed the field extension of $\R$ into its algebraic closure $\R\lgw \C$ is the most simple one.

\item When we want to understand the algebraic closure of $\C(x_1)$, the field of rational functions in one variable,
we equip $\C(x_1)$ with the norm $\|\cdot\|$ defined by 
$$\forall p, q\in \C[x_1],\ \ \|p/q\|:=e^{-\ord(p)+\ord(q)}$$
and we study the algebraic closure of $\C(x_1)$ into its completion $\C(\!(x_1)\!)$. Indeed, by the Newton-Puiseux Theorem, the field extension of $\C(\!(x_1)\!)$ into its algebraic closure, the field of Puiseux series, is well described.
\end{enumerate}
It is fascinating that there are similar results between these situations in spite of the fact that the technics used to prove them are quite different. For instance, there is an analogue of the Liouville diophantine approximation Theorem for the elements of $\L^\o$ that are algebraic over $\K(\!(x)\!)$ (see \cite{Ro}, \cite{I-I}, \cite{Hi}). There is also an analogue of  Eisenstein's Theorem \cite{Ei} for the elements of $\L^\o$ that are algebraic over $\K(\!(x)\!)$ (see \cite[Theorem 5.12]{Ro2}) and an analogue of Fabry's Theorem for the elements of $\L^\o$ that are algebraic over $\K(\!(x)\!)$ (see \cite[Theorem 6.4]{AR}).\\
\\
In this paper we investigate necessary conditions for a Laurent series with support in a rational strongly convex cone to be algebraic over $\K(\!(x)\!)$ in any characteristic. We provide conditions in terms of the support of the series.  Indeed in the case of the study of the algebraic closure of $\C(x_1)$ into $\C(\!(x_1)\!)$, or the algebraic closure of $\K(x_1)$ into $\K(\!(x_1)\!)$ for a general field $\K$, such conditions have been given, and some questions remain open (as the Dynamical Mordell-Lang Conjecture -  cf. \cite{BHS} or \cite{BGT}). \\
In order to explain this we introduce the following definition:
 
\begin{defi}
Let  $\xi$  be a series with support in $\Q^n$ and coefficients in a field $\K$. We set
	$$\t(\xi):= \left\{ \omega\in {\R_{\geq 0}}^n\mid  \exists k \in\R,\ \Supp (\xi )\cap \left\{  u\in\R^n\mid u\cdot\omega\leq k\right\}=\emptyset\right\}.$$
\end{defi}
For example, if $\Supp(\xi)$ is equal to a cone $\s$ and every unbounded face of $\s$ contains infinitely many elements of $\Supp(\xi)$, then $\t(\xi)^\vee=\s$ (see Definition \ref{cone_def} for the dual of a cone). Let us mention that we restrict to vectors $\o\in{\R_{\geq 0}}$ since, for a series $\xi$ algebraic over $\K(\!(x)\!)$, $\xi+f(x)$ is algebraic over $\K(\!(x)\!)$ for any $f(x)\in\K[[x]]$.
\begin{rem}\label{convex_cone}
It is straightforward to check that $\t(\xi)$ is a (non necessarily polyhedral) convex cone (see Lemma \ref{convex}).
\end{rem}
Our first main result is that $\t(\xi)$ is rational when $\xi$ is algebraic over $\K(\!(x)\!)$:

\begin{thm}\label{main_thm}
Let $\xi$ be a Laurent Puiseux series whose support is included in a translation of a strongly convex cone containing ${\R_{\geq 0}}^n$ and with coefficients in a characteristic zero field $\K$. Assume that $\xi$ is algebraic over $\K(\!(x)\!)$.  Then the set $\tau(\xi)$ is  a strongly convex rational cone.

\end{thm}

From the rationality of $\t(\xi)$ we can deduce easily the following result:

\begin{cor}\label{cor_main}
Let $\xi$ be a Laurent Puiseux series whose support is included in a translation of a strongly convex cone containing ${\R_{\geq 0}}^n$ and with coefficients in a characteristic zero field $\K$. Assume that $\xi$ is algebraic over $\K(\!(x)\!)$.  Then there is $\g\in\Z^n$ such that
$$\Supp(\xi)\subset \g+\t(\xi)^\vee.$$
Moreover $\t(\xi)^\vee$ is the smallest (non necessarily polyhedral) cone having this property.
\end{cor}

Now the question is to determine how far is the support of $\xi$ of being equal to a set of the form $\g+\t(\xi)^\vee$. The following result provides an answer to this question:
 \begin{thm}\label{main_thm3}
 Let $\xi$ be a Laurent Puiseux series whose support is included in a translation of a strongly convex cone containing ${\R_{\geq 0}}^n$ and with coefficients in a characteristic zero field $\K$.  Assume that $\xi$ is algebraic over $\K(\!(x)\!)$.  Then there exist a finite set $C\subset\Z^n$,  a Laurent polynomial $p(x)$, and a power series $f(x)\in\K[[x]]$  such that
 $$\Supp(\xi+p(x)+ f(x))\subset C+\t(\xi)^\vee$$
and for every unbounded facet $F$ of 
$\Conv(C+\t(\xi)^\vee)$,
we have
$$\#\left\{\Supp(\xi+p(x)+ f(x))\cap F\right\}=+\infty.$$
 \end{thm}

We will see in Example \ref{ex_C} that, in general,  the set $C$  cannot be chosen to be one single point. We will also see in Example \ref{ex_min} that there is no minimal, maximal or canonical $C$ satisfying Theorem \ref{main_thm3}.\\
We do not know if this statement can be extended to faces of $\t(\xi)^\vee$ of smaller dimension. But we have the following result:
\begin{thm}\label{main_thm2}
Let $\xi$ be a Laurent Puiseux series whose support is included in a translation of a strongly convex cone containing ${\R_{\geq 0}}^n$ and with coefficients in a  field $\K$ of  characteristic zero. Assume that $\xi$ is algebraic over $\K(\!(x)\!)$. Then, for every  $u\in\R_{>0}^n$ in the boundary of $\t(\xi)$  there exists a Laurent polynomial $p(x)$ such that, if $F_u$ denotes the face defined by $u$ of the convex hull of
$\Supp(\xi+p(x))$, then 
$$\#\left(F_u\cap\Supp(\xi)\right)=+\infty.$$
\end{thm}

Let us mention that the cone $\t(\xi)$ was already considered in \cite{AR} where we were not able to prove its rationality and where we gave a very much weaker version of Theorem \ref{main_thm2}.\\
\\
We will begin by  the proof of Theorem \ref{main_thm}. This proof is not very difficult once we have
 the right setting, and is essentially based on two tools: the compacity of the space of orders on ${\R_{\geq 0}}^n$, and the construction, for every  order $\preceq$ on $\Q^n$, of an algebraically closed field $\mathcal S^\K_\preceq$ containing $\K(\!(x)\!)$. This  result of compacity is  due to Ewald and Ishida \cite{EI} (see also  \cite{T}) and is a purely topological result.  It will allow us to have a decomposition  of ${\R_{\geq0}}^n$ into a union of finitely many rational strongly convex cones having the following property:  for each order $\preceq$, the roots of the minimal polynomial of $\xi$ in $\mathcal S_\preceq^\K$ have support in the dual of one of these cones.\\
 The construction of the algebraically closed fields $\mathcal S^\K_\preceq$ has been given in \cite{AR}  and is based on systematic constructions  of algebraically closed valued fields due to Rayner \cite{ra}. \\
The proofs of  Theorems \ref{main_thm3} and \ref{main_thm2} are much more involved. First they require  the introduction of intermediate cones that we have to describe and compare with $\t(\xi)$. 
Then we need to prove an extension of Dickson's Lemma for general rational cones (see Proposition \ref{int_cones}) that will help us to show the existence of the finite set $C$ of Theorem \ref{main_thm2}. \\
Finally we investigate the positive characteristic case. We begin by constructing algebraically closed fields containing $\K(\!(x)\!)$. Each of these fields depends on an order $\preceq$ on $\Q^n$, and their definition extends the definition of $\mathcal S_\preceq^\K$ to the case of a positive characteristic field $\K$. Then we provide several examples showing that Theorems \ref{main_thm3} and \ref{main_thm2} as long as Proposition \ref{prop_key}, that is the key tool to prove Theorem \ref{main_thm}, are no longer true in the positive characteristic case.
\\
\\
The authors are very grateful to the referee, who made a great work helping the authors to clarify the  paper. They also thank Diane MacLagan who brought to their attention a mistake in a previous version of this work.


\section{Orders and algebraically closed fields containing $\K(\!(x)\!)$}
In this section we introduce the tools needed for the proof of Theorem \ref{main_thm}.


\subsection{The space of orders on ${\R_{\geq0}}^n$}
\begin{defi}\label{cone_def}
Let us recall that a \emph{cone} $\t\subset \R^n$ is a subset of $\R^n$ such that for every $t\in\t$ and $\la\geq 0$, $\la t\in\t$. 
A cone $\t\subset \R^n$ is \emph{polyhedral} if it has the form
$$\t=\{\la_1u_1+\cdots+\la_s u_s\mid \la_1,\ldots,\la_s\geq 0\}$$
for some given vectors $u_1$, \ldots, $u_s\in\R^n$. A cone is said to be a \emph{rational cone}  if it is polyhedral, and the $u_i$ can be chosen in $\Z^n$.\\
A cone is \emph{strongly convex} if it does not contain any non trivial linear subspace.\\
In practice, as almost all the cones that we consider in this paper are polyhedral cones, the term cone will always refer to polyhedral cones (unless stated otherwise).\\ 
The \emph{dual} $\sigma^{\vee}$ of a cone $\sigma$ is the cone given by 
$$\sigma^{\vee} :=  \{ v \in \R^n \mid v \cdot u \geq 0,\, \text{for all}\, u\in \sigma \}$$
where $u\cdot v$ stands for the dot product $(u_1,\ldots ,u_n)\cdot (v_1,\ldots ,v_n):= u_1v_1+\cdots +u_nv_n$.
\end{defi}

\begin{rem}\label{basic_tau}
Let $\xi$ be a series and $\o\in\t(\xi)$. Then $\Supp(\xi)\subset \g+\lg \o\rg^\vee$ for some $\g\in\Z^n$. Indeed it is enough to choose $\g$ such that  $\Supp(\xi)\cap\{u\in\R^n\mid u\cdot \o\leq \g\cdot \o\}=\emptyset$.
\end{rem}

\begin{defi}
A \emph{preorder} on an abelian group $G$ is a binary relation $\preceq$ such that
\begin{itemize}
\item[i)] $\forall u, v\in G$, $u\preceq v$ or $v\preceq u$,
\item[ii)] $\forall u,v,w\in G$, $u\preceq v$ and $v\preceq w$ implies $u\preceq w$,
\item[iii)] $\forall u,v,w\in G$, $u\preceq v$ implies $u+w\preceq v+w$,

\end{itemize}
The set of preorders on $G$ is denoted by $\ZR(G)$. The set of orders on $G$ is a subset of $\ZR(G)$ denoted by $\Ord(G)$.
\end{defi}

\begin{defit}\label{robbiano}
By \cite[Theorem 2.5]{R}\label{rob} for every $\preceq\in\ZR(\Q^n)$ there exist an integer $s\geq 0$ and orthogonal vectors $u_1$, \ldots, $u_s\in{\R}^n$ such that
$$\forall u,v\in \Q^n,\ u\preceq v \Longleftrightarrow (u\cdot u_1, \ldots, u\cdot u_s)\leq_{\text{lex}} (v\cdot u_1,\ldots, v\cdot u_s).$$
For such a preorder we set $\preceq\,:=\,\leq_{(u_1,\ldots, u_s)}$. Such a preorder extends in an obvious way to a preorder on $\R^n$ and the preorders of this form are called \emph{continuous preorders}.\\
We remark that the orthogonality condition is not essential as, if $U_j$ denotes the linear subspace generated by  $u_1$, \ldots, $u_{j-1}$, and $v_j$ is chosen in $u_j+U_j$ for every $j\geq 2$, then $\leq_{(u_1,\ldots, u_s)}=\leq_{(u_1,v_2,\ldots, v_s)}$.
\end{defit}

\begin{defi}
Let $A\subset\R^n$ and $\preceq$ be a continuous preorder on $\R^n$. We say that $A$ is $\preceq$-positive if 
$$\forall a\in A, \ \ a \succeq \underline 0.$$
\end{defi}

\begin{defi}
Let $\preceq$ be a continuous preorder on $\R^n$ and $A\subset \R^n$. We say that $A$ is $\preceq$-well-ordered if $A$ is well-ordered with respect to $\preceq$.
\end{defi}

\begin{defi}
The set of continuous orders $\preceq$ such that  ${\R_{\geq 0}}^n$ is $\preceq$-positive is denoted by $\Ord_n$.
\end{defi}


\begin{defi}
Given two preorders $\preceq_1$ and $\preceq_2$, one says that \emph{$\preceq_2$ refines $\preceq_1$} if
$$\forall u,v\in\R^n,\ u\preceq_2 v\Longrightarrow u\preceq_1 v.$$
\end{defi}

\begin{rem}\label{ref_rem}
Let $(u_1,\ldots, u_s)$ be nonzero vectors of $\R^n$. Using  Theorem-Definition \ref{robbiano} it is easy to check that for a preorder $\preceq$, $\preceq$ refines $\leq_{(u_1,\ldots, u_s)}$ if and only if there exist vectors $u_{s+1}$, \dots, $u_{s+k}$ such that $\preceq=\leq_{(u_1,\ldots,u_{s+k})}$.
\end{rem}

\begin{lem}\label{int_posi}
Let $\o\in\R^n$ and $\s$ be a strongly convex cone with $\o\in\Int(\s^\vee)$. Then $\s$ is $\preceq$-positive for every order $\preceq$ refining $\leq_\o$.
\end{lem}
\begin{proof}
If $\o\in\Int(\s^\vee)$, we have that $s\cdot\o>0$ for every $s\in\s\setminus\{0\}$. By Theorem-Definition \ref{ref_rem}, every $\preceq$ refining $\leq_\o$ is equal to $\leq_{(\o,v_1,\ldots, v_s)}$ for some vectors $v_i$. Thus $\s$ is $\preceq$-positive.
\end{proof}

The next easy lemma will be used several times:

\begin{lem}\label{inter_cones}\cite[Lemma 2.4]{AR}
Let  $\sigma_1$ and $\sigma_2$ be two cones and $\g_1$ and $\g_2$ be vectors of $\R^n$. Let us assume that $\s_1\cap\s_2$ is full dimensional. Then there exists a vector $\g\in\Z^n$ such that

	$$\left( \g_1 +\sigma_1\right)\cap\left( \g_2 +\sigma_2\right)\subset \g + \sigma_1\cap\sigma_2.$$
\end{lem}

Finally we give the following result, which will be used in the proof of Theorem \ref{main_thm2} (this is a generalization of \cite[Corollary 3.10]{AR}):

\begin{lem}\label{lem_union_cones}
Let $\s_1$, \ldots, $\s_N$ be strongly convex cones and let $\o\in\R^n\setminus\{\underline 0\}$. The following properties are equivalent:
\begin{itemize}
\item[i)] We have
$\o\in\Int\left(\bigcup_{i=1}^N\s_i^\vee\right)$. 
\item[ii)] For every order $\preceq\in\Ord(\Q^n)$ refining $\leq_\o$, there is an index $i$ such that $\s_i$ is $\preceq$-positive.
\end{itemize}
\end{lem}

\begin{proof}
Let us prove that i) implies ii). Let $\o\in \Int\left(\bigcup_{i=1}^N\s_i^\vee\right)$. We are going to show that for all nonzero vectors $v_1$, \ldots, $v_{n-1}\in\lg \o\rg^\perp$, with $v_j\in\lg \o,v_1,\ldots, v_{j-1}\rg^\perp$ for every $j$, there is an integer $i$ such that $\s_i$ is $\leq_{(\o,v_1,\ldots, v_{n-1})}$-positive. Indeed,  by Remark \ref{ref_rem} every preorder  refining $\leq_\o$ is of the form $\leq_{(\o,v_1,\ldots, v_{j})}$ for $1\leq j\leq n-1$. Therefore ii) is satisfied. So from now on, we fix such  vectors $v_1$, \ldots, $v_{n-1}$.\\
By Lemma \ref{int_posi}, if $\o\in\Int(\s_i^\vee)$ for some $i$, then $\s_i$ is $\preceq$-positive for every $\preceq$-refining $\leq_\o$. In particular it is $\leq_{(\o,v_1,\ldots, v_{n-1})}$-positive.
 Otherwise, let $E_1$ denote the set of indices $i$ such that $\o\in\s_i^\vee$. If $\o$ were in the boundary of $\bigcup_{i\in E_1} \s_i^\vee$, then $\o$ would belong to some $\s_i$ for $i\notin E_1$ because $\o\in \Int\left(\bigcup_{i=1}^N\s_i^\vee\right)$. Thus  $\o\in \Int\left(\bigcup_{i\in E_1} \s_i^\vee\right)$.\\
 Since $\o\in\Int\left(\bigcup_{i\in E_1} \s_i^\vee\right)$, there is $\la_1>0$  such that $\o+\la_1 v_1 \in \Int\left(\bigcup_{i\in E_1} \s_i^\vee\right)$. 
Then two cases may occur:\\ 
(1) Assume $\o+\la_1 v_1\in\Int(\s_i^\vee)$ for some $i\in E_1$. Because $i\in E_1$,   for $s\in \s_i\setminus\{0\}$, either $\o\cdot s>0$, or $\o\cdot s=0$. In this last case we have $v_1\cdot s>0$ since $(\o+\la_1 v_1)\cdot s>0$ and $\la_1 >0$. Therefore  $\s_i$ is $\preceq$-positive for every order $\preceq$ refining $\leq_{(\o,v_1)}$ (In particular it is $\leq_{(\o,v_1,\ldots, v_{n-1})}$-positive).\\
(2) If $\o+\la_1 v_1\notin \Int(\s_i^\vee)$ for every $i\in E_1$, we denote by $E_2$ the set of $i \in E_1$ such that $\o+\la_1 v_1\in \s_i^\vee$. As before we necessarily have $\o+\la_1 v_1\in\Int\left(\bigcup_{i\in E_2} \s_i^\vee\right)$.
Therefore there is $\la_2>0$  such that $\o+\la_1 v_1+\la_2 v_2 \in \Int\left(\bigcup_{i\in E_2} \s_i^\vee\right)$. Once again, if $\o+\la_1 v_1+\la_2 v_2 \in\Int(\s_i^\vee)$ for some $i\in E_2$, $\s_i$ is $\preceq$-positive for every order $\preceq$ refining $\leq_{(\o,v_1,v_2)}$. Otherwise we repeat the same process until one of the two situations occurs:
\begin{enumerate}
\item[a)] there is $j<n-1$ such that  $\o+\la_1v_1+\cdots+\la_j v_j\in \Int(\s_i^\vee)$ for some $i$. Then, we can prove in the same way as (1) that $\s_i$ is $\preceq$-positive for every $\preceq$ refining $\leq_{(\o,v_1,\ldots, v_{j})}$ (hence it is $\leq_{(\o,v_1,\ldots, v_{n-1})}$-positive).  
\item[b)] there is no such an index $j$. Thus we end with $\o+\la_1 v_1+\cdots+\la_{n-1} v_{n-1}$ that belongs to (at least) one $\s_i^\vee$. Therefore the cone $\s_i$ is $\leq_{(\o,v_1,\ldots, v_{n-1})}$-positive, because $\o\in\s_i^\vee$, $\o+\la_1v_1 \in\s_i^\vee$, \ldots, $\o+\la_1 v_1+\cdots+\la_{n-1}v_{n-1}\in\s_i^\vee$.
\end{enumerate}
This proves that i) implies ii).\\
Now we prove the converse.  Assume that for every order $\preceq\in\Ord(\Q^n)$ refining $\leq_\o$, there is an index $i$ such that $\s_i$ is $\preceq$-positive.\\
 Let $v$ be a vector with $\|v\|=1$. By assumption, there is an index $i$ such that $\s_i$ is $\leq_{(\o,v)}$-positive. Let $s_1$, \ldots, $s_l$ be generators of $\s_i$ that we assume to be of norm equal to 1. Reordering the $s_j$, there is an integer $k\geq 0$ such that 
$s_j\cdot \o>0$ for every $j\leq k$, and $s_j\cdot \o=0$ for every $j>k$, because $\s_i$ is $\leq_{(\o,v)}$-positive. Take $\la>0$. When $k>1$ assume moreover that  $\dfrac{\min_{\ell\leq k}\{s_\ell\cdot\o\}}{2}\geq \la$. Then we claim that $\o+\la v\in\s_i^\vee$. Indeed, if $j\leq k$ we have
$$(\o+\la v)\cdot s_j=\o\cdot s_j +\la v\cdot s_j\geq \o\cdot s_j-\la \|v\| \|s_j\|\geq \dfrac{\min_{\ell\leq k}\{s_\ell\cdot\o\}}{2}>0.$$
If $j>k$ we have
$$(\o+\la v)\cdot s_j=\la v\cdot s_j \geq 0$$
since $\s_i$ is $\leq_{(\o,v)}$-positive. This implies that $\o+\la v\in\s_i^\vee$. Since this is true for every $v$, we have $\o\in\Int\left(\bigcup_{i=1}^N\s_i^\vee\right)$. 
\end{proof}

\begin{cor}\label{prop_ne}
Let $\o\in{\R_{\geq 0}}^n\backslash\{\underline 0\}$ and let $\s_1$, \ldots, $\s_N$ be strongly convex  cones which are $\leq_\o$-positive. Assume that for every order $\preceq\in\Ord_n$ refining $\leq_\o$, there is an index $i$ such that $\s_i$ is $\preceq$-positive. Then there is a neighborhood $V$ of $\o$ such that, for every $\o'\in V$ and every $\preceq'\in\Ord_n$ refining $\leq_{\o'}$, there is an index $i$ such that $\s_i$ is $\preceq'$-positive.

\end{cor}
\begin{proof}
We have $\o\in\Int\left(\bigcup_{i=1}^N\s_i^\vee\right)$ by the previous lemma. Therefore, the previous lemma shows that we can choose $V=\Int\left(\bigcup_{i=1}^N\s_i^\vee\right)$.
\end{proof}

The following lemma will be used several times:

\begin{lem}\label{values_support} Let $\xi$ be a Laurent series with coefficients in a field $\K$. Assume that $\Supp(\xi)\subset \g+\s$ where $\g\in\Z^n$ and $\s$ is a rational cone. Let $\o\in\s^\vee$. Then, for every $t\in\R$, the set
$$\{u\cdot \o\mid u\in \Supp(\xi)\}\cap \left]-\infty,t\right]$$
is finite.
\end{lem}
\begin{proof} We can make a translation and assume that $\g=0$. Since $\s$ is a rational cone,  by Gordan's Lemma, there exist vectors $v_1$, \ldots, $v_N\in\s\cap\Z^n$ generating  $\s\cap\Z^n$  as a monoid. Since $\o\in\s^\vee$, we have $v_i\cdot \o\geq 0$ for every $i$.\\
By assumption we have $\s=\left\{\sum_{i=1}^Nn_iv_i\mid n_i\in\N\right\}$. Therefore the set
$\{u\cdot \o\mid u\in \Supp(\xi)\}$ is included in the monoid generated by $v_1\cdot\o$, \ldots, $v_N\cdot \o$. Since this monoid is finitely generated, the sets $\{u\cdot \o\mid u\in \Supp(\xi)\}\cap \left]-\infty,t\right]$ are finite.
\end{proof}

\subsection{The space $\Ord_n$ as a compact topological space}
One important tool for the proof of Theorem \ref{main_thm} is the fact that the set of orders $\Ord_n$ is a topological compact space for a well chosen topology. This topology has been introduced by Ewald and Ishida \cite{EI} (see also \cite{DR} for a generalization of this to the sets of preorders on a given group).

\begin{defi}\cite{EI}\cite{T}
The set $\ZR(\Q^n)$  is endowed with a topology for  which the sets
$$\mathcal{U}_{\sigma}:=\left\lbrace \preceq \in \ZR(\Q^n) \text{ such that } \sigma \text{ is } \preceq \text{-positive}\right\rbrace$$
form a basis of open sets where $\s$ runs over the full dimensional strongly convex rational cones.
\end{defi}

\begin{rem}
With this definition we have $\Ord_n=\mathcal U_{{\R_{\geq 0}}^n}\cap \Ord(\Q^n)$.
\end{rem}

We have the following result:

\begin{thm}\label{comp}\cite{EI}
The space $\ZR(\Q^n)$ is compact and $\Ord(\Q^n)$ is closed in $\ZR(\Q^n)$. Moreover every $\mathcal U_\s$ is compact. Therefore $\Ord_n$ is compact.
\end{thm}
The following lemma will be useful in the sequel:

\begin{lem}\label{compact_alt}
Let $\s_1$, \ldots, $\s_N$ be rational cones such that 
$\Ord_n\subset \bigcup_{k=1}^N\mathcal U_{\s_k}.$
Then
$${\R_{\geq 0}}^n\subset\bigcup_{k=1}^N\s_k^\vee.$$ 
\end{lem}
\begin{proof}
Let $\o\in {\R_{\geq 0}}^n$. Let $\preceq\in \Ord_n$ refining $\leq_\o$. Such a $\preceq$ exists by \cite[Lemma 3.18]{AR}. Then $\preceq\in\mathcal U_{\s_k}$ for some $k$. Since $\preceq$ refines $\leq_\o$, we have that $\s_k$ is $\leq_\o$-positive. This means that $\o\in\s_k^\vee$. This proves that ${\R_{\geq 0}}^n\subset\bigcup_{k=1}^N\s_k^\vee$.
\end{proof}



\subsection{Algebraically closed fields containing $\K(\!(x)\!)$ in characteristic zero}

\begin{defi}\label{def_field_0} Let $n$ be a positive integer and $\preceq\in \Ord_n$.\\
For a field $\K$  of characteristic zero, we denote by $\mathcal S_\preceq^\K$ the following set:
$$\left\{ \xi\text{ series}\mid \exists k\in \N^*, \gamma\in\Z^n,  \sigma \preceq\text{-positive rational cone, } 		\Supp  (\xi )\subset (\gamma +\sigma )\cap \frac{1}{k}\Z^n \right\}.$$
\end{defi}
We have the following theorem:

\begin{thm}\label{thm_alg_clos}\cite[Theorem 4.5]{AR}
When $\K$ is an algebraically closed field of characteristic zero, the set $\mathcal S_\preceq^\K$ is an algebraically closed field.
\end{thm}

\begin{defi}\label{tech_def}
For simplicity we will use the following notation: given a characteristic zero field $\K$, a strongly convex rational cone $\s$ containing ${\R_{\geq 0}}^n$ and $k\in\N^\ast$, we set
\begin{equation*}\SS^\K_{\s,k}:=\left\{  \vphantom{\frac{1}{kp^l}}\xi\text{ series }\mid  \exists   \gamma\in\Z^n, \text{ such that } 
\Supp  (\xi )\subset (\gamma +\sigma )\cap \frac{1}{k}\Z^n \right\}. \end{equation*}
\end{defi}


\section{Proofs of Theorem \ref{main_thm}  and Corollary \ref{cor_main}}
We begin by the following remark:

\begin{lem}\label{convex}
Let $\xi$ be a series with support in $\Q^n$. Then $\t(\xi)$ is a convex (non necessarily polyhedral) cone. 
\end{lem}

\begin{proof}
It is straightforward to see that for $\la>0$ and $\o\in \t(\xi)$, $\la \o\in\t(\xi)$. Thus, we need to prove that for $\o_1$, $\o_2\in\t(\xi)$, $\o_1+\o_2\in\t(\xi)$. By Remark \ref{basic_tau}, there exist $\g_1$, $\g_2\in\Z^n$, $\s_1\subset \langle \o_1\rangle^\vee$, $\s_2\subset \langle \o_2\rangle^\vee$ containing ${\R_{\geq 0}}^n$ such that $\Supp(\xi)\subset (\g_1+\s_1)\cap(\g_2+\s_2)$. Thus $\Supp(\xi)\subset \g+\s_1\cap\s_2$ for some $\g\in\Z^n$ by Lemma \ref{inter_cones}. This proves the lemma.
\end{proof}

In order to prove Theorem   \ref{main_thm} we need the following intermediate results:
\begin{lem}\label{key_lem} Let $\K$ be a characteristic zero field. Let $\xi\in \SS^\K_{\s,k}$ where $\s$ is a strongly convex rational cone containing ${\R_{\geq 0}}^n$, and $k\in\N^\ast$ (cf. Definition \ref{tech_def}). 
Let $P\in\K[[x]][T]$ be a monic polynomial of degree $d$ with $P(\xi)=0$. Let us assume that there exists
 $\s_0\supset {\R_{\geq 0}}^n$ a strongly convex rational cone such that $P(T)$ splits in $\SS^\K_{\s_0,k}$.\\
 Then
$$\Int(\s_0^\vee)\cap \tau(\xi)\neq\emptyset\Longrightarrow \s_0^\vee\subset \tau(\xi).$$ 
\end{lem}

\begin{proof}
 Consider a nonzero vector $\o\in\Int(\s_0^\vee)\cap\t(\xi)$. Since $\xi\in \SS^\K_{\s,k}$, there are $k\in\N$, $\g_0\in\Z^n$, and $ \sigma$ a $\leq_\o$-positive rational cone,  such that 
$$\Supp  (\xi )\subset (\gamma_0 +\sigma )\cap \frac{1}{k}\Z^n.$$
 Since $\s$ is $\leq_\o$-positive and strongly convex, there exists an order $\preceq\in\Ord_n$ refining $\leq_\o$ such that $\s$ is $\preceq$-positive (see \cite[Lemma 3.8]{AR}).  Thus $\xi$ is a root of $P$ in $\mathcal S_{\preceq}^\K$.\\
On the other hand, $\o$ is in the interior of $\s_0^\vee$, so $\s_0$ is $\preceq$-positive by Lemma \ref{int_posi}. Thus the roots of $P$ in $\SS^\K_{\s_0,k}$ are the roots of $P$ in $\SS^\K_{\preceq}$ and $\xi$ is one of them. Hence there is some $\g\in\Z^n$ such that
$$\Supp(\xi)\subset\g+\s_0.$$
Now let  $\o'\in\s_0^\vee$. We have
$\s_0\subset \lg \o'\rg^\vee.$
Hence $\o'\in\t(\xi)$. This proves the lemma.
\end{proof}


\begin{cor}\label{cor}
Let $\K$ be an algebraically closed field of characteristic zero and 
let $\xi\in\SS^\K_\preceq$ where $\preceq\in\Ord_n$. Let $P\in\K[[x]][T]$ be a monic polynomial of degree $d$ with $P(\xi)=0$.
Let $\s_i$, $i=1,\ldots, N$, be  strongly convex  rational cones containing ${\R_{\geq 0}}^n$, and $k\in\N^\ast$, satisfying the following properties:
\begin{itemize}
\item[i)] $\displaystyle\bigcup_{i=1}^N \s_i^\vee={\R_{\geq 0}}^n$,
\item[ii)] for every $i$, and every $\preceq\in\mathcal U_{\s_i}$, the roots of $P(T)$ in $\SS^\K_\preceq$ are in $\SS^\K_{\preceq,\s_i,k}$.
\end{itemize}
Then, after renumbering the $\s_i$, there is an integer $l\leq N$ such that 
$$\t(\xi)=\bigcup_{i=1}^l\s_i^\vee.$$
\end{cor}

\begin{proof}
By Lemma \ref{key_lem}, we can renumber the $\s_i$ such that $\s_i^\vee\subset \t(\xi)$ for $i\leq l$ and $\Int(\s_i^\vee)\cap\t(\xi)=\emptyset$ for every $i>l$. So we have $\displaystyle\bigcup_{i=1}^l\s_i^\vee\subset\t(\xi)$.\\
Now, suppose that this inclusion is strict: there is an element $\o\in\t(\xi)$ such that $\displaystyle\o\notin\bigcup_{i=1}^l\s_i^\vee$.  \\
We claim that  $\displaystyle\bigcup_{i=1}^l\s_i^\vee$ is convex. Indeed, assume that it is not. Since the  $\s_i^\vee$ are convex, this implies that there is $\o_{i_1}\in\s_{i_1}^\vee$, $\o_{i_2}\in\s_{i_2}^\vee$ for some $i_1$, $i_2\leq l$, such that $\o_{i_1}+\o_{i_2}\notin \displaystyle\bigcup_{i=1}^l\s_i^\vee$. In this case, the line segment $[\o_{i_1},\o_{i_2}]$ intersects $\displaystyle\bigcup_{i=l+1}^N\s_i^\vee$. Since $\s_{i_1}^\vee$ and $\s_{i_2}^\vee$ are full dimensional, we can replace freely $\o_{i_1}$ and $\o_{i_2}$ by any elements close to them. Thus we may assume that $[\o_{i_1},\o_{i_2}]$ intersects $\Int(\s_m^\vee)$ for some $m>l$. But this contradicts the fact that $\t(\xi)$ is convex (see Lemma \ref{convex}).\\
Therefore, by the Hahn-Banach Theorem there is a hyperplane $H$ separating $\o$ and the convex closed set $\displaystyle\bigcup_{i=1}^l\s_i^\vee$ in the following sense: one open half space delimited by $H$, denoted by $O$, contains $\o$ and $\displaystyle\bigcup_{i=1}^l\s_i^\vee\subset \R^n\backslash \ovl O$. Since $\displaystyle\bigcup_{i=1}^l\s_i^\vee$ is full dimensional, the convex hull $\mathcal C$ of $\o$ and $\displaystyle\bigcup_{i=1}^l\s_i^\vee$ is full dimensional:
 $$\mathcal C:=\left\{\la\o+(1-\la )v\mid v\in  \bigcup_{i=1}^l\s_i^\vee, 1\geq\la\geq 0\right\}.$$
  Thus $\mathcal C\cap O$ contains an open ball $B$. \\
Since $\t(\xi)$ is convex (see Lemma \ref{convex}), $\mathcal C \subset \t(\xi)$ and $B\subset \t(\xi)$. Then $B$ intersects one $\s_m^\vee$ for $m>l$ because $B \subset O$ and we have assumed $\displaystyle\bigcup_{i=1}^N \s_i^\vee={\R_{\geq 0}}^n$. But because $B$ is open, $B\cap\Int(\s_m^\vee)\neq \emptyset$, and this is a contradiction because $B\subset \t(\xi)$ and $\t(\xi)\cap\Int(\s_m^\vee)=\emptyset$ for $m>l$. Therefore the inclusion is not strict and $\displaystyle\bigcup_{i=1}^l\s_i^\vee=\t(\xi)$.
\end{proof}
\begin{prop}\label{prop_key}
Let $\K$ be an algebraically closed field of characteristic zero and $P\in\K[[x]][T]$. There is an integer $N$, strongly convex rational cones $\s_1$, \ldots, $\s_N$ containing ${\R_{\geq 0}}^n$, and $k\in\N^\ast$, such that:
\begin{itemize}
\item[i)] $\Ord_n\subset \bigcup_{i=1}^N\mathcal U_{\s_{i}}$ and $\displaystyle\bigcup_{i=1}^N \s_i^\vee={\R_{\geq 0}}^n$,
\item[ii)] for every $\preceq\in\Ord_n$, there is $j\in\{1,\ldots, N\}$, such that the roots of $P(T)$ in $\SS_\preceq^\K$ belong to $\SS^\K_{\s_j,k}$.
\end{itemize}
\end{prop}

\begin{proof}
By Theorem \ref{thm_alg_clos} for every order $\preceq\in \Ord_n$ there is an element  $\g_{\preceq}\in \Z^n$, and  a $\preceq$-positive  strongly convex rational cone $\s_{\preceq}$ such that the roots of $P$ can be expanded as  series in $\mathcal S_{\preceq}^\K$ with support in $\g_{\preceq}+\s_{\preceq}$.  \\
In particular we have $\Ord_n\subset \mathcal U_{{\R_{\geq 0}}^n}\subset \bigcup_{\preceq}\mathcal U_{\s_\preceq}$. Hence, by Theorem \ref{comp}, we can extract from this family of cones $\s_\preceq$, a finite number of cones, denoted by $\s_{\preceq_1}$, \ldots, $\s_{\preceq_N}$, such that $\Ord_n\subset \bigcup_{i=1}^N\mathcal U_{\s_{\preceq_i}}$. Therefore, by Lemma \ref{compact_alt}, we have that
${\R_{\geq 0}}^n\subset\displaystyle\bigcup_{i=1}^N \s_{\preceq_i}^\vee$. Because the $\s_{\preceq_i}$ contain ${\R_{\geq 0}}^n$, we have 
${\R_{\geq 0}}^n=\displaystyle\bigcup_{i=1}^N \s_{\preceq_i}^\vee.$
Moreover these cones satisfy the following property:
\begin{itemize}
\item[] $\forall \preceq\in\Ord_n$, $\exists \g_\preceq\in\Z^n$, $\exists i\in\{1,\ldots, N\}$, such that the roots of $P(T)$ in $\SS^\K_\preceq$ have support in $\g_\preceq+\s_{\preceq_i}$.
 \end{itemize}
Assume that the same integer $i\in\{1,\ldots, N\}$ satisfies the previous property for two orders
$\preceq_1$ and $\preceq_2\in\Ord_n$. That is,  the roots of $P$ in $\mathcal S_{\preceq_1}^\K$ (resp. in $\mathcal S_{\preceq_2}^\K$) have support in $\g_{\preceq_1}+\s_i$ (resp. in $\g_{\preceq_2}+\s_i$). Then the roots of $P$ in $\mathcal S_{\preceq_2}^\K$ are elements of $\mathcal S_{\preceq_1}^\K$, thus the roots of $P$ in $\mathcal S_{\preceq_2}^\K$ coincide with its roots in $\mathcal S_{\preceq_1}^\K$. Therefore we may assume that the element $\g_\preceq$  does depend only on $i$.
\end{proof}

\begin{proof}[Proof of Theorem \ref{main_thm}] 
First, by replacing each of the $x_i$ by some power of  $x_i$, we may assume that $\xi$ is a Laurent series.
By Proposition \ref{prop_key}, there exist strongly convex  rational cones $\s_1$, \ldots, $\s_N$ satisfying i) and ii) of Corollary \ref{cor}. Therefore, by Corollary \ref{cor}, we have that $\t(\xi)$ is a strongly convex rational cone.  This proves Theorem \ref{main_thm}.
\end{proof}

\begin{rem}
For a formal power series $f\in\K[[x]]$ we denote by $\NP(f)$ its Newton polyhedron.
Let $p$ be a vertex of $\NP(f)$. The set of vectors $v\in\R^n$ such that $p+\lambda v\in \NP(f)$ for some  $\lambda\in \R_{\geq0}$ is a rational strongly convex cone. Such a cone  is called the \emph{cone of the Newton polyhedron of $f$ associated with the vertex $p$}.  We have the following generalization of Abhyankar-Jung Theorem that provides in an effective way some cones satisfying Corollary \ref{cor}:
\end{rem}
\begin{thm}[Strong form of the Abhyankar-Jung Theorem] \cite[Th\'eor\`eme 3]{GP}\cite[Theorem 7.1]{A}\cite[Theorem 6.2]{PR}\label{AB}
Let $\K$ be a characteristic zero field. Let $P(Z)\in\K[[x]][Z]$ be a monic polynomial and let $\D$ be its discriminant. Let $\NP(\D)$ denote the Newton polyhedron of $\D$. Then the set of cones of $\NP(\D)$ satisfies the properties of Corollary \ref{cor}.
\end{thm}
Therefore, if $\xi$ is integral over $\K[[ x]]$, that is $P(T)$ is a monic polynomial in $T$, we may replace the use of Corollary \ref{cor} (thus Proposition \ref{prop_key} and thus Theorems \ref{comp} and \ref{thm_alg_clos}) by Theorem \ref{AB}.

Now we are able to  prove  Corollary \ref{cor_main}:

\begin{proof}[Proof of Corollary \ref{cor_main}]
By replacing $\K$ by its algebraic closure, we may assume that $\K$ is algebraically closed.
Since $\t(\xi)$ is rational, let $\o_1$, \ldots, $\o_s\in\Z^n$ be generators of $\t(\xi)$. Thus we have $\t(\xi)^\vee=\bigcap_{i=1}^s \langle \o_i\rangle^\vee$. Therefore, we have $\Supp(\xi)\subset \g+\t(\xi)^\vee$ for some $\g\in\Z^n$ by  Remark \ref{basic_tau} and Lemma \ref{inter_cones}.\\
On the other hand if $\s$ is a cone (not necessarily finitely generated) such that 
$\Supp(\xi)\subset \g+\s$ for some $\g\in\Z^n$, then we have $\s^\vee\subset \t(\xi)$ by the definition of $\t(\xi)$, that is, $\t(\xi)^\vee\subset \s$.
\end{proof}


\section{Proof of Theorems \ref{main_thm3} and \ref{main_thm2}}\label{section4}
%
%
%
\subsection{Preliminary results}
\begin{defi}
For a Laurent series $\xi$ we set
$$\tau'_0(\xi )= \left\{ \omega\in {\R_{\geq 0}}^n\setminus\{\underline0\}\mid \# \left(\Supp (\xi )\cap \left\{  u\in\R^n\mid u\cdot\omega\leq k\right\}\right)<\infty, \forall k\in\R\right\},$$
$$\tau'_1(\xi )= \left\{ \omega\in {\R_{\geq 0}}^n\setminus\{\underline0\}\mid \# \left(\Supp (\xi )\cap \left\{  u\in\R^n\mid u\cdot\omega\leq k\right\}\right)=\infty, \forall k\in\R\right\}.$$
\end{defi}
We have the following lemma:

\begin{lem}\label{compar}
Let $\xi$ be a Laurent series with support in a translation of a strongly convex cone containing ${\R_{\geq 0}}^n$.
We have $\t'_0(\xi)\subset\t(\xi)\subset\ovl{\t'_0(\xi)}$.
\end{lem}

\begin{proof}
 We have $\t'_0(\xi)\subset \t(\xi)$ by definition.\\
  Let $\o\in\t(\xi)$. Then by Remark \ref{basic_tau}, $\Supp(\xi)\subset \g+\lg \o\rg^\vee$ for some $\g\in\Z^n$.\\
   On the other hand, by hypothesis, $\Supp(\xi)$ is included in $\g'+\sigma$ where $\g'\in\Z^n$ and $\sigma$ is a strongly convex cone such that  ${\R_{\geq 0}}^n\subset \sigma$. Thus,  by Lemma \ref{inter_cones}, $\Supp(\xi)$ is included in a translation of the strongly convex cone $\sigma\cap\lg \o\rg^\vee$.\\ We have
   $\o \in {\lg \o\rg^\vee}^\vee\subset \left(\sigma\cap\lg \o\rg^\vee \right)^\vee,$
    and $\left(\sigma\cap\lg \o\rg^\vee \right)^\vee$ is full dimensional. Thus there exists a sequence $(\o_k)_k$ of vectors in $ \mathrm{Int}\left(\left(\sigma\cap\lg \o\rg^\vee \right)^\vee\right)$  that converges to $\o$.\\
  We have to prove that the $\o_k$ belong to $\t'_0(\xi)$. For $u\in (\s\cap\lg \o\rg^\vee)\setminus\{\underline0\}$, we have $u\cdot\o_k\neq 0$ because $\o_k\in\Int\left(\left(\sigma\cap\lg \o\rg^\vee \right)^\vee\right)$. This shows that $\sigma\cap\lg \o\rg^\vee\cap\lg \o_k\rg^\perp=\{\underline 0\}$. Therefore, because $\Supp(\xi)$ is included in a translation of $\s\cap\lg\o\rg^\vee$, for all $k$ we have:
   $$\o_k\in  \left\{ \omega'\in {\R}^n\mid \# \left(\Supp (\xi )\cap \left\{  u\in\R^n\mid u\cdot\omega'\leq k\right\}\right)<\infty, \forall k\in\R\right\}.$$
    Moreover, because $\o\in\t(\xi)\subset {\R_{\geq 0}}^n$ and ${\R_{\geq 0}}^n\subset \sigma$, we have ${\R_{\geq 0}}^n=({\R_{\geq 0}}^n)^\vee  \subset \sigma\cap\lg \o\rg^\vee$. Therefore the $\o_k$ are in ${\R_{\geq 0}}^n$, and they are nonzero for $k$ large enough because $(\o_k)_k$ converges to $\o$ which is nonzero.
This shows that $\o_k\in\t'_0(\xi)$ for $k$ large enough, therefore $\o\in\ovl{\t'_0(\xi)}$.
\end{proof}

\begin{cor}\label{tau_clos} Under the hypothesis of Theorem \ref{main_thm2}, we have 
$$\t(\xi)=\ovl{\t'_0(\xi)}.$$
\end{cor}

\begin{proof}
By Lemma \ref{compar} we have $\t'_0(\xi)\subset \t(\xi)\subset \ovl{\t'_0(\xi)}$. Since $\t(\xi)$ is closed (it is a rational cone, thus a polyhedral cone, by Theorem \ref{main_thm}) we have $\t(\xi)=\ovl{\t'_0(\xi)}$. 
\end{proof}

\begin{defi}
In the rest of this section we consider the following setting: $\xi$ is a Laurent series with support included in the translation of a strongly convex rational cone, and $\xi$ is  algebraic over $\K[[x]]$ where $\K$ is a  field of characteristic zero. From now on we enlarge $\K$ in order to  assume that $\K$ is algebraic closed. We denote by $P\in \K[[x]][T]$  the minimal polynomial of $\xi$ and, for any  order $\preceq\in\Ord_n$,  $\xi_1^\preceq$,  \ldots, $\xi_d^\preceq$ denote the roots of $P(T)$ in $\mathcal{S}_\preceq^\K$.  We set
$$\tau_0 (\xi ):= \left\{\omega\in {\R_{\geq 0}}^n\backslash\{\underline0\}\mid \text{ for all }\preceq\text{ that refines }\leq_\omega ,\ \exists i \text{ such that } \xi = \xi_i^\preceq\right\},$$
$$\tau_1 (\xi ):= \left\{\omega\in {\R_{\geq 0}}^n\backslash\{\underline0\}\mid\xi \neq\xi_i^\preceq, \text{ for all }\preceq\text{ that refines }\leq_\omega, \ \forall i= 1,\ldots ,d \right\},$$
\end{defi}

\begin{rem}\label{rem_rap}
These sets were introduced in \cite{AR}, but only for $\o\in{\R_{>0}}^n$. In this case it was proved that $\t_0(\xi)\cap{\R_{>0}}^n=\t'_0(\xi)\cap{\R_{>0}}^n$ and $\t_1(\xi)\cap{\R_{>0}}^n=\t'_1(\xi)\cap{\R_{>0}}^n$ (see \cite[Lemmas 5.8, 5.11]{AR}). Taking into account all the $\o\in{\R_{\geq 0}}^n$ changes the situation. In particular we do not have $\t_0(\xi)=\t'_0(\xi)$ in general (see Example \ref{bad_ex}).
\end{rem}


\begin{prop}\label{prop_eq}
We have $\t_1(\xi)=\t'_1(\xi)$ and $\t'_0(\xi)\subset\t_0(\xi)$.
\end{prop}

\begin{proof}
The proof of the equality $\t_1(\xi)=\t'_1(\xi)$ is exactly the proof of \cite[Lemma 5.11]{AR}.
Let us prove $\t'_0(\xi)\subset\t_0(\xi)$. Let $\o\in\t'_0(\xi)$, in particular:
\begin{equation}\label{eq}
	\#\left(\Supp (\xi )\cap \left\{  u\in\R^n\mid u\cdot\omega\leq k\right\}\right)<\infty, \ \forall k\in\R,
\end{equation}
and let us consider an order $\preceq$ that refines $\leq_\o$.\\
Let $(u_l)_l$ be a sequence of elements of $\Supp(\xi)$ such that $u_l\succeq u_{l+1}$ for every $l\in\N$. Then $u_l\geq_\o u_{l+1}$, that is $u_l\cdot\o \geq u_{l+1}\cdot\o$, for every $l\in\N$. Therefore by (\ref{eq}), this sequence contains only finitely many distinct terms. Therefore $u_{l+1}=u_l$ for $l$ large enough because $\preceq$ is an order. This shows that $\Supp (\xi)$ is $\preceq$-well-ordered. Thus by \cite[Corollary 4.6]{AR} $\xi$ is an element of $\mathcal{S}_\preceq^\K$. This shows that $\o\in\t_0(\xi)$.
\end{proof}


\begin{prop}\label{prop_open}
The sets $\t_0(\xi)$ and $\t_1(\xi)$ are open subsets of ${\R_{\geq 0}}^n$.\end{prop}

\begin{proof}
Let us consider the cones $\s_i$ given by  Proposition \ref{prop_key}. In particular, for every $\o\in{\R_{\geq 0}}^n\backslash\{\underline 0\}$, the set of orders $\preceq\in\Ord_n$ refining $\leq_\o$ is included in
  $\bigcup_{i=1}^N\mathcal U_{\s_i}$. The set $\mathcal T_\o=\{\s_{1},\ldots, \s_{N}\}$ satisfies the following property:
\begin{itemize}
\item[] For any order $\preceq\in\Ord_n$ refining $\leq_\o$, there is $\s\in\mathcal T_\o$, $\s$ being $\preceq$-positive, such that the roots of $P$ in $\mathcal S_\preceq^\K$ are in $\SS_{\s,k}^\K$ for some $k\in\N^\ast$.
\end{itemize}
 Moreover, let us  choose $\mathcal T_\o$ to   be minimal among the sets of cones having this property.  Then Corollary \ref{prop_ne} implies that, for every $\o'\in{\R_{\geq 0}}^n\backslash\{\underline 0\}$ close enough to $\o$, and  for any order $\preceq'\in\Ord_n$ refining $\leq_{\o'}$, there is $\s\in\mathcal T_\o$ such that the roots of $P$ in $\mathcal S_{\preceq'}^\K$ are in $\SS^\K_{\s,k}$ for some $k\in\N^\ast$. Since $\mathcal T_\o$ is minimal with this property, for every $\o'$ close enough to $\o$, for every order $\preceq'\in\Ord_n$ refining $\leq_{\o'}$ and for every $i=1,\ldots, d$, there is an order $\preceq\in\Ord_n$ refining $\leq_\o$  such that $\xi_i^{\preceq'}=\xi_{j_i}^{\preceq}$ for some $j_i$.\\
If $\o\in\t_0(\xi)$ then $\xi$ is equal to some $\xi_i^\preceq$ for every order $\preceq\in\Ord_n$ refining $\leq_\o$. Thus, for every $\o'\in{\R_{\geq 0}}^n$ close enough to $\o$ and every order $\preceq'\in\Ord_n$ refining $\leq_{\o'}$, $\xi=\xi_{j}^{\preceq'}$ for some $j$. Thus $\o'\in\t_0(\xi)$. This proves that $\t_0(\xi)$ is open in ${\R_{\geq 0}}^n$.\\
If $\o\in\t_1(\xi)$ then $\xi\neq\xi_i^\preceq$ for every $i$ and for every order $\preceq\in\Ord_n$ refining $\leq_\o$. Thus, for $\o'\in{\R_{\geq 0}}^n$ close enough to $\o$ and every order $\preceq'\in\Ord_n$ refining $\leq_{\o'}$, $\xi\neq\xi_{j}^{\preceq'}$ for every $j$. Hence $\o'\in\t_1(\xi)$ and $\t_1(\xi)$ is open.
\end{proof}


\begin{cor}\label{cor1}
We have
$$\ovl{\t'_0(\xi)}\cap\t'_1(\xi)=\emptyset.$$
\end{cor}

\begin{proof}
The sets $\t_0(\xi)$ and $\t_1(\xi)$ are disjoint and open in  ${\R_{\geq 0}}^n$. Thus $\ovl{\t_0(\xi)}\cap\t_1(\xi)=\emptyset$. This proves the corollary because $\t'_0(\xi)\subset \t_0(\xi)$ and $\t'_1(\xi)=\t_1(\xi)$ by Proposition \ref{prop_eq}.
\end{proof}


\begin{lem}\label{lem_clos}
We have
$$\ovl{\t'_0(\xi)}=\ovl{\t_0(\xi)\cap{\R_{>0}}^n}=\ovl{\t_0(\xi)}.$$
\end{lem}

\begin{proof}
The set $\t_0(\xi)$ is open. Therefore every $w\in\t_0(\xi)\cap({\R_{\geq 0}}^n\backslash{\R_{>0}}^n)$ can be approximated by elements of $\t_0(\xi)\cap{\R_{>0}}^n$.
Hence 
$$\ovl{\t_0(\xi)\cap{\R_{>0}}^n}=\ovl{\t_0(\xi)}.$$
By \cite[Lemma 5.8]{AR} $\t'_0(\xi)\cap{\R_{>0}}^n=\t_0(\xi)\cap{\R_{>0}}^n$. We have that $\t'_0(\xi)$ is convex (the proof is exactly the same as the proof of \cite[Lemma 5.9]{AR}). Thus we have
$$\ovl{\t'_0(\xi)\cap{\R_{>0}}^n}=\ovl{\t'_0(\xi)}$$
by \cite[Prop. 16 - Cor. 1; II.2.6]{B}.
Hence
$$\ovl{\t'_0(\xi)}=\ovl{\t'_0(\xi)\cap{\R_{>0}}^n}=\ovl{\t_0(\xi)\cap{\R_{>0}}^n}=\ovl{\t_0(\xi)}.$$
\end{proof}

\begin{cor}\label{cone_inv}
For every $f\in\K[[x]]^*$ we have
$$\t_0(\xi+f)=\t_0(\xi),\ \t_1(\xi+f)=\t_1(\xi),\ \t(\xi+f)=\t(\xi),$$
$$\t_0(f\xi)\supset\t_0(\xi),\ \t_1(f\xi)\supset\t_1(\xi),\ \t(f\xi)\supset\t(\xi).$$
\end{cor}\label{lem_inv}
\begin{proof}
The minimal polynomial of $\xi+f$ is $Q(T):=P(T-f)$. Thus, for a given $\preceq\in\Ord_n$, the roots of $Q(T)$ in $\mathcal S_\preceq^\K$ are $\xi_1^\preceq+f$, \ldots, $\xi_d^\preceq+f$. This shows that
$$\t_0(\xi+f)=\t_0(\xi),\ \t_1(\xi+f)=\t_1(\xi).$$
 Lemma \ref{lem_clos} and Corollary \ref{tau_clos} imply that
$\t(\xi+f)=\t(\xi).$\\
Now, the polynomial $R(T):=f^dP(T/f)$ vanishes at $f\xi$. On the other hand, if $\ovl R(T)$ is a polynomial with $\ovl R(f\xi)=0$, then $\ovl R(fT)$ is a polynomial vanishing at $\xi$. This shows that $P(T)$ divides $\ovl R(fT)$. Thus, the minimal polynomial of $f\xi$ has degree $d$ and divides $R(T)$, thus it is of the form $\frac{1}{g}R(T)=\frac{f^d}{g}P(T/f)$ for some $g\in\K[[x]]$, $g\neq 0$.\\
Therefore, for a given $\preceq\in\Ord_n$, the roots in $\mathcal S_\preceq^\K$ of the minimal polynomial of $f\xi$  are $f\xi_1^\preceq$, \ldots, $f\xi_d^\preceq$. This shows that 
$$\t_0(f\xi)\supset\t_0(\xi),\ \t_1(f\xi)\supset\t_1(\xi),\ \ovl{\t'_0(f\xi)}\supset\ovl{\t'_0(\xi)}.$$
This proves the corollary.
\end{proof}

\begin{ex}
Let $\xi=\sum_{i\geq 0}\left(\frac{x_1}{x_2}\right)^i=\frac{x_2}{x_2-x_1}$. Here $\t(\xi)$ is the cone generated by $(1,0)$ and $(1,1)$. But $\t((x_2-x_1)\xi)={\R_{\geq 0}}^n$. Therefore we do not have $\t(f\xi)=\t(\xi)$ in general.
\end{ex}

\begin{ex}\label{bad_ex}
We can see on a basic example that  $\t'_0(\xi+f)\neq \t'_0(\xi)$ in general: let $n=2$ and fix $\xi=\sum_{k\in\N}x_1^k$ and $f=1-\xi$. Then $\t'_0(\xi)=\R_{> 0}\times\R_{\geq 0}$ but $\t'_0(\xi+f)={\R_{\geq 0}}^2$. This also shows that $\t_0(\xi)\neq \t'_0(\xi)$ in general.
\end{ex}



\subsection{Proof of Theorem \ref{main_thm2}}
 First, by replacing each of the $x_i$ by some power of  $x_i$, we may assume that $\xi$ is a Laurent series.\\
Here we denote by $\s$ the face of $\t(\xi)^\vee$ defined by $u$. We set
\begin{equation}\label{hyper}H_u(t):=\{v\in\R^n\mid v\cdot u=t\},\ H_u(t)^+=\{v\in\R^n\mid v\cdot u\geq t\}.\end{equation}
The vector $u$ is in the boundary of $\t'_0(\xi)$  because $\t(\xi)=\ovl{\t'_0(\xi)}$ by Corollary \ref{tau_clos}. Hence  by Corollary \ref{cor1} we have $u\notin \t'_1(\xi)$. Thus, we have $u\in\t_0'(\xi)$ or $u\in{\R_{\geq 0}}^n\backslash (\t_0'(\xi)\cup\t'_1(\xi))$. Assume that $u\in\t'_0(\xi)$. By Proposition \ref{prop_open},  $\t'_0(\xi)\cap{\R_{>0}}^n$ is open. Thus, because $u$ is in the boundary of $\t'_0(\xi)$, we have $u\in{\R_{\geq 0}}^n\backslash {\R_{>0}}^n$, which contradicts the hypothesis. Therefore $u\notin \t'_0(\xi)$. Thus we use the following lemma whose proof is given below:

\begin{lem}\label{lem_faces}
Let $u\notin  \t'_0(\xi)\cup\t'_1(\xi)$. Then there exist a Laurent polynomial $p_\s(x)$  and a real number $t_\s$ such that
\begin{equation}\label{bound}\Supp(\xi+p_\s(x))\subset H_u(t_\s)^+ \text{ and } \#\left(\Supp(\xi+p_\s(x))\cap H_u(t_\s)\right)=+\infty.\end{equation}
\end{lem}

 Now we denote by $p(x)$ the sum of distinct monomials that appear in all the $\p_\s(x)$ (there is a finite number of faces $\s$), and Theorem \ref{main_thm2} is proved.

\begin{proof}[Proof of Lemma \ref{lem_faces}]
Because $u\notin \t'_0(\xi)\cup\t'_1(\xi)$, 
the following set is non empty and bounded from above :
$$E_\s:=\left\{ t\in\R\mid \# \left(\Supp (\xi )\cap \left\{  v\in\R^n\mid v\cdot u< t\right\}\right)<\infty\right\}.$$
Let us set $t_\s:=\sup E_\s.$ 
By Lemma \ref{values_support}, the set $\{v\cdot u\mid v\in\Supp(\xi)\}\cap]-\infty, t]$ is finite for every $t$ (here $u$ belongs to the closure of $\t(\xi)$). Thus, we may order the elements of $\{v\cdot u\mid v\in\Supp(\xi)\}$ as $t_0<t_1<\cdots$, and necessarily $t_\s$ is one of these elements. Therefore the set $\Supp (\xi )\cap \left\{  v\in\R^n\mid v\cdot u= t_\s\right\}$ is infinite and  $\Supp (\xi )\cap \left\{  v\in\R^n\mid v\cdot u< t_\s\right\}$ is finite.  So we denote by $-p_\s(x)$ the sum of the monomials of $\xi$ whose exponents belong to $\{v\in\R^n\mid v\cdot u<t_\s\}$
and \eqref{bound} is satisfied (that is, we remove from $\xi$ the monomials that are in $\left\{  v\in\R^n\mid v\cdot u< t_\s\right\}$). 
\end{proof}

\subsection{Proof of Theorem \ref{main_thm3}}
We begin by giving  a strengthened version of Lemma \ref{inter_cones} that we will need in the proof of Theorem  \ref{main_thm3}:

\begin{prop}[Dickson's Lemma]\label{int_cones}
Let $\s_1$, \ldots, $\s_k$ be convex rational cones such that $\s:=\bigcap_{j=1}^k\s_j$ is a full dimensional convex rational cone. Let $\g_1$, \ldots, $\g_k\in  \Z^n$. Then there exists a finite set $C\subset\Z^n$ such that
$$\bigcap_{j=1}^k(\g_j+\s_j)\cap\Z^n=C+\s\cap\Z^n.$$
\end{prop}

\begin{proof}
Up to a translation we may assume that $\g_j\in \s\cap\Z^n$ for every $j$ because $\s$ is full dimensional.
Let $u_1$, \ldots, $u_s$ be vectors with  integer coordinates generating $\s\cap\Z^n$. Then the ring $R_\s$ of polynomials in $x_1$, \ldots, $x_n$ with support in $\s\cap\Z^n$ is isomorphic to  $\K[U_1,\ldots, U_s]/I$ for some binomial ideal $I$. This is well known and this can be described as follows (for instance see \cite[Proposition 1.1.9]{CLS} for details):\\
 for any linear relation $L:=\{\sum_{i=1}^s \la_iu_i=0\}$ with $\la_i\in\Z$ we consider the binomial 
$$B_L:=\prod_{i\mid\la_i\geq 0} U_i^{\la_i}-\prod_{i\mid \la_i< 0}U_i^{-\la_i}.$$ Then $I$ is the ideal generated by the $B_L$ for $L$ running over the $\Z$-linear relations between the $u_i$. Moreover, for $\g\in\s\cap\Z^n$, the isomorphism $R_\s\lgw\K[U]/I$  sends $x^\g$ onto $U^{\a_\g}$ where $\a_\g\in\Z^s_{\geq 0}$ is defined by $\g=\sum_{i=1}^s\a_{\g,i}u_i$.\\
Because the $\g_j$ belong to $\s$, we have
$$\bigcap_{j=1}^k(\g_j+\s_j)\subset \bigcap_{j=1}^k\s_j=\s.$$
Therefore the set of monomials $x^u$ for $u\in \bigcap_{j=1}^k(\g_j+\s_j)\cap\Z^n$, is equal to the set of monomials of a monomial ideal of $R_\s$. By Gordan's Lemma, this ideal is generated by a finite number of monomials. If $C$ denotes the set of exponents of these generators, we have
$\bigcap_{j=1}^k(\g_j+\s_j)\cap\Z^n=C+\s\cap\Z^n.$
\end{proof}

\begin{proof}[Proof of Theorem \ref{main_thm3}]
As in the proof of Theorem \ref{main_thm2} we may replace each of the $x_i$ by some power of  $x_i$, and assume that $\xi$ is a Laurent series.\\
By \cite[Proposition 1.3]{Oda}, because $\t(\xi)^\vee$ is a strongly convex rational cone, for each nonzero face $\s\subset \t(\xi)^\vee$, there is a vector $u_\s$ in the boundary of $\t(\xi)$ such that
$$\s=\lg u_\s\rg^\perp\cap \t(\xi)^\vee.$$
In fact, as seen in the proof of \cite[Proposition 1.3]{Oda}, we can freely choose $u_\s$ in the relative interior of $\s^\perp\cap\t(\xi)$, where $\s^\perp\cap\t(\xi)$ is a face of dimension $n-\dim(\s)$ of $\t(\xi)$. Thus, when $\s$ is a facet of  $\t(\xi)^\vee$, $\s^\perp\cap\t(\xi)$ is a half-line that is generated by one vector with integer coordinates. Therefore, when $\s$ is a facet of $\t(\xi)^\vee$, we can choose $u_\s\in\Z^n$. \\
From now on, $\s$ will always denote a facet of $\t(\xi)^\vee$.
We have
$$\t(\xi)^\vee=\bigcap_{\s\text{ facet of } \t(\xi)^\vee}H_{u_\s}(0)^+$$
where the $H_{u_\s}(t)^+$ are defined in \eqref{hyper}.
The vectors $u_\s$ are in the boundary of $\t'_0(\xi)$  because $\t(\xi)=\ovl{\t'_0(\xi)}$ by Corollary \ref{tau_clos}. Hence  by Corollary \ref{cor1} we have $u_\s\notin \t'_1(\xi)$ for any facet $\s$. Thus for every facet $\s$ of $\t(\xi)^\vee$ we have $u_\s\in\t_0'(\xi)$ or $u_\s\in{\R_{\geq 0}}^n\backslash (\t_0'(\xi)\cup\t'_1(\xi))$. We will reduce to the situation where none of the $u_\s$ are in $\t_0'(\xi)$:\\
Let $\s$ be a facet of $\t(\xi)^\vee$ for which $u_\s\in\t'_0(\xi)$. By Proposition \ref{prop_open},  $\t'_0(\xi)\cap{\R_{>0}}^n$ is open. Thus, because $u_\s$ is in the boundary of $\t'_0(\xi)$, we have $u_\s\in{\R_{\geq 0}}^n\backslash {\R_{>0}}^n$. In particular  at least one of the coordinates of $u_\s$  is zero,   hence $\lg u_\s\rg^\perp$ contains at least one line generated by one vector with integer coordinates. Therefore there exists $f_\s(x)\in\K[[x]]$ with support  in $\lg u_\s\rg^\perp\cap {\R_{\geq 0}}^n$ and such that 
$$\#\left\{\Supp(\xi+f_\s(x))\cap \lg u_\s\rg^\perp\cap {\R_{\geq 0}}^n\right\}=+\infty.$$
Moreover we can do this simultaneously for every facet $\s$ of $\t(\xi)^\vee$ such that $u_\s\in\t'_0(\xi)$, hence there exists $f(x)\in\K[[x]]$ such that for every such facet $\s$:
\begin{equation}\label{eq_t}\#\left\{\Supp(\xi+f(x))\cap \lg u_\s\rg^\perp\cap {\R_{\geq 0}}^n\right\}=+\infty.\end{equation}
By Corollary \ref{cone_inv} $\t(\xi)=\t(\xi+f(x))$. But $u_\s\notin\t_0'(\xi+f(x))$ by \eqref{eq_t}. Therefore, we  replace $\xi$ with $\xi+f(x)$. This does not change $\t(\xi)$, but this allows us to assume that  $u_\s\in{\R_{\geq 0}}^n\backslash (\t_0'(\xi)\cup\t'_1(\xi))$. Therefore we may assume that none of the $u_\s$ is in $\t_0'(\xi)$.\\
\\
Then we apply Lemma \ref{lem_faces} to see, as in the proof of Theorem \ref{main_thm2}, that
modulo a finite number of monomials and a formal power series $f(x)\in\K[[x]]$, the support of $\xi$ is included in $\displaystyle\bigcap_{\s\text{ facet of }\t(\xi)^\vee}H_{u_\s}(t_\s)^+\cap\Z^n$. Moreover  each $H_{u_\s}(t_\s)$ contains infinitely many monomials of $\xi$, i.e there is a Laurent polynomial $p(x)$ such that
$$\Supp(\xi+f(x)+p(x))\subset \bigcap_{\s\text{ facet of }\t(\xi)^\vee}H_{u_\s}(t_\s)^+\cap\Z^n$$
$$\text{ and } \#\left(\Supp(\xi+f(x)+p(x))\cap H_{u_\s}(t_\s)\right)=+\infty\ \ \ \forall \s.$$
 For every $\s$  facet of $\t(\xi)^\vee$ we have $H_{u_\s}(t_\s)^+=\g_\s+H_{u_\s}(0)^+$ for  any $\g_\s\in H_\s(t_\s)$. But, since $\displaystyle H_{u_\s}(t_\s)\cap\Z^n\neq\emptyset$, we may fix  $\g_\s\in\Z^n$. Since $u_\s\in\Z^n$, the cone $H_{u_\s}^+$ is rational. Thus, by Corollary \ref{int_cones} there is a finite set $C\subset \Z^n$ such
 that $$\bigcap_{\s\text{  facet  of }\t(\xi)^\vee}H_{u_\s}(t_\s)^+\cap\Z^n=C+\bigcap_{\s\text{  facet  of }\t(\xi)^\vee}H_{u_\s}(0)^+\cap \Z^n=C+\t(\xi)^\vee\cap \Z^n.$$ 
Because the sum of two convex sets is a convex set, we have 
$$\Conv(C+\t(\xi)^\vee)=\Conv(C)+\t(\xi)^\vee$$
is an unbounded convex polytope. 
Moreover each unbounded facet of $\Conv(C+\t(\xi)^\vee)$ is the intersection of $\Conv(C+\t(\xi)^\vee)$ with one $H_{u_\s}$ for some facet $\s$ of $\t(\xi)^\vee$. Therefore every unbounded facet of $\Conv(C+\t(\xi)^\vee)$ contains 
 infinitely many elements of $\Supp(\xi+f(x)+p(x))$. 
\end{proof}


\section{Some examples}

\begin{ex}\label{ex1}
Let $E:=\{(x,y)\in\R_{\geq 0}\times \R\mid y\geq-x-\sqrt{x}\}$ and let $\xi$ be a Laurent series whose support is $\Z^2\cap E$ as follows: \tdplotsetmaincoords{60}{120} 
\begin{center}\begin{figure}[H]\fbox{\begin{tikzpicture} [scale=1.1, axis/.style={->,thin}, 
vector/.style={-stealth,red,very thin}, 
vector guide/.style={dashed,red,thick}]

\coordinate (O) at (0,0);

\pgfmathsetmacro{\ax}{0.8}
\pgfmathsetmacro{\ay}{0.8}

\coordinate (P) at (\ax,\ay);

\draw[axis] (0,0) -- (2.3,0) node[anchor= west]{$x$};
\draw[axis] (0,0,0) -- (0,1.3) node[anchor=north west]{$y$};

\draw[domain=0:1,smooth,variable=\x] plot (\x,{-\x-sqrt(\x)});
 
  \fill [color=gray, opacity=0.1,domain=0:1,smooth,variable=\x]
plot (\x,{-\x-sqrt(\x)})-- (2,{-1-sqrt(1)}) -- (2,1) -- (0,1) -- cycle;
\draw[axis] (0,0) -- (2.3,0) node[anchor= west]{$x$};
\end{tikzpicture}}\caption{Example \ref{ex1}}\label{fig_ex1}\end{figure}\end{center}
Here $\t(\xi)$ is the cone generated by $(1,0)$ and $(1,1)$, but $\Supp(\xi)\subset\s$ where $\s$ is the cone $\{(x,y)\in\R_{\geq 0}\times \R\mid y>-x\}$. Since $\s\subsetneq \t(\xi)^\vee$, $\xi$ is not algebraic over $\K(\!(x,y)\!)$ by Corollary \ref{cor_main}.\\
Moreover $\t'_1(\xi)$ is equal to the rational cone generated by $(1,0)$ and $(1,1)$ minus the origin. So $\t'_1(\xi)$ is not open. 
In this case  ${\R_{\geq 0}}^n\backslash\{\underline 0\}=\t'_0(\xi)\cup\t'_1(\xi)$.
\end{ex}


\begin{ex}\label{ex_2}
We consider the set
$$E:=\{(x,y)\in\R_{\geq 0}\times \R\mid y\geq \ln(x+1)\}.$$ We rotate it by an angle of $-\pi/4$ and denote  this set by $\G$.
We denote  a Laurent series whose support is $\G\cap\Z^2$ by $\xi$ (see Figure \ref{fig_ex_2}).

\tdplotsetmaincoords{60}{120} 
\begin{center}\begin{figure}[H]\fbox{\begin{tikzpicture} [scale=1.1, axis/.style={->,thin}, 
vector/.style={-stealth,red,very thin}, 
vector guide/.style={dashed,red,thick}]

\coordinate (O) at (0,0);

\pgfmathsetmacro{\ax}{0.8}
\pgfmathsetmacro{\ay}{0.8}

\coordinate (P) at (\ax,\ay);

\draw[axis] (0,0) -- (2.3,0) node[anchor= west]{$x$};
\draw[axis] (0,0,0) -- (0,1.3) node[anchor=north west]{$y$};

 \draw plot [smooth] coordinates {(0,0) (0.2,-0.03) (0.4,-0.065) (0.6,-0.115) (0.8,-0.18) (1,-0.26) (1.2,-0.36) (1.4,-0.48) (1.7, -0.67) (2,-0.9)};
 
 \fill [color=gray, opacity=0.1]
(0,0) -- (2,2)
-- plot [smooth] coordinates {(0,0) (0.2,-0.03) (0.4,-0.065) (0.6,-0.115) (0.8,-0.18) (1,-0.26) (1.2,-0.36) (1.4,-0.48) (1.7, -0.67) (2,-0.9)}
-- (2.2,-0.9)-- (2.2,1) -- (1,1) -- cycle;
\draw[axis] (0,0) -- (2.3,0) node[anchor= west]{$x$};

\end{tikzpicture}}\caption{Example \ref{ex_2}}\label{fig_ex_2}\end{figure}\end{center}
Then $\t(\xi)^\vee$ is the cone generated by $(1,-1)$ and $(0,1)$, so it is rational, but $\xi$ is not algebraic as Theorem \ref{main_thm3}  is not satisfied.\\
Moreover $\t(\xi)$ is generated by $(1,0)$ and $(1,1)$. Thus the vector $(1,1)$ is in the boundary of $\t(\xi)$ but here $(1,1)\in\t'_0(\xi)$. Thus $\t'_0(\xi)$ is closed.
\end{ex}


\begin{ex}\label{ex_min}
Let $\s$ be the cone generated by the vectors $(1,0)$, $(0,1)$ and $(1,-1)$. Then the series $\xi:=\sum_{k=0}^\infty (xy^{-1})^k$ has support in $\s$ and it is straightforward to see that $\s=\t(\xi)^\vee$. Let $N\in\Z^*$ and set $p_N(x,y):=\sum_{k=0}^{N-1}(xy^{-1})^k$ (when $N> 0$) or $p_N(x,y)=\sum_{k=N}^0(xy^{-1})^k$ (when $N<0$). Let $C_N$ denote the point $(N,-N)$. Then, we have 
$$C_N\in\Supp(\xi-p_N(x,y))\subset C_N+\s.$$
This shows that there is no canonical choice for $C_N$ in Theorem \ref{main_thm3}, neither a minimal or maximal $C_N$.
\end{ex}
\begin{ex}\label{ex_C}
Let $C$ be the set $\{(1,0,0), (0,1,0), (0,0,1)\}$, and let $\s$ be the cone generated by the vectors  $(1,-1,1)$, $(-1,1,1)$,  and $(1,1,-1)$.
We can construct a Laurent series $\xi$, algebraic over $\K[[x,y,z]]$, with support in $\Conv(C)+\s$, such that all the unbounded faces of $\Conv(C)+\s$ contain infinitely many monomials of $\xi$ as follows:\\
We fix an algebraic series $G(T)$ not in $\K(T)$. We remark that, for $a$, $b$, $c\in\Z$, the series $G(x^ay^bz^z)$ is algebraic over $\K(x,y,z)$, and it is a formal sum of monomials of the form $x^{ka}y^{kb}z^{kc}$ with $k\in\N$. Thus its support is included in the half line generated by the vector $(a,b,c)$.\\
Then we set 
$$\xi=G(x)+G(y)+zG(z)+zG\left(\frac{xz}{y}\right)+zG\left(\frac{yz}{x}\right)+(x+y)G\left(\frac{xy}{z}\right).$$
Then $\xi$ is algebraic over $\K(\!(x,y,z)\!)$, its support is $\Conv(C)+\s$ and all the unbounded faces of $\Conv(C)+\s$ contain infinitely many monomials of $\xi$ (see Figure \ref{fig_ex_C}).  Therefore $\t(\xi)^\vee=\s$. Moreover we can  see that there is no $\g\in\R^n$ such that $\Supp(\xi)\subset \g+\s$ and every face of $\g+\s$ contains infinitely many monomials of $\xi$, even after removing monomials of $\xi$ belonging to ${\R_{\geq 0}}^3$. Indeed, if it were the case, the four unbounded edges of $\Conv(C)+\s$ that are not included in ${\R_{\geq 0}}^3$ would intersect at one point and this is clearly not the case. Thus we cannot assume that the finite set $C$ of Theorem \ref{main_thm3}  is a single point.

\tdplotsetmaincoords{60}{120} 
\begin{center}\begin{figure}[H]\fbox{\begin{tikzpicture} [scale=1, tdplot_main_coords, axis/.style={->,thin}, 
vector/.style={-stealth,red,very thin}, 
vector guide/.style={dashed,red,thick}]

\coordinate (O) at (0,0,0);

\pgfmathsetmacro{\ax}{0.8}
\pgfmathsetmacro{\ay}{0.8}
\pgfmathsetmacro{\az}{0.8}

\coordinate (P) at (\ax,\ay,\az);

\draw[axis] (0,0,0) -- (2.3,0,0) node[anchor= east]{$x$};
\draw[axis] (0,0,0) -- (0,2.3,0) node[anchor=north west]{$y$};
\draw[axis] (0,0,0) -- (0,0,2.3) node[anchor=south]{$z$};

\draw[thick] (1,0,0) -- (0,0,1);
\draw[thick] (1,0,0) -- (2,1,-1);
\draw[thick] (0,1,0) -- (1,2,-1);
\draw[thick]  (0,1,0) -- (0,0,1);
\draw[thick]  (1,0,0) -- (0,1,0);
\draw[thick] (1,0,0) --(2,0,0);
\draw[thick] (0,1,0) --(0,2,0);
\draw[thick] (0,0,1) -- (1,-1,2);
\draw[thick] (0,0,1) -- (-1,1,2);
\draw[thick](0,0,1)-- (0,0,2);

\end{tikzpicture}}\caption{Example \ref{ex_C}}\label{fig_ex_C}\end{figure}\end{center}

\end{ex}

\section{The positive characteristic case}
 In positive characteristic, unlike the characteristic zero case, we cannot express roots of polynomials as Puiseux series with support in rational strongly convex cones. This already appears in the univariate case, since it has been noticed by Chevalley \cite{Ch} that none of the roots of the polynomial $T^p-x_1^{p-1}T-x_1^{p-1}$ can be expressed as Puiseux series, when $p>0$ denotes the characteristic of the base field. This shows that the  Newton-Puiseux Theorem is no more valid in positive characteristic.   Then Abhyankar noticed that for such a polynomial, the roots can be expressed as generalized series with support in $\Q$ with the additional property that their support is well-ordered \cite{Ab}. Here such a root can be written as
$\displaystyle \sum_{k\in\N^*} x_1^{1-\frac{1}{p^k}}.$
The determination of the algebraic closure of $\K(\!(x_1)\!)$ for $n=1$, when $\K$ is a positive characteristic field, was finally achieved very recently (see \cite{K1}, \cite{K2}).\\
For $n\geq 2$, this problem has recently been investigated by Saavedra \cite{S}.
He generalized Macdonald's Theorem to the positive characteristic case as follows:

\begin{thm}\cite[Theorem 5.3]{S}\label{saavedra}
Let $\K$ be an algebraically closed field of characteristic $p> 0$. Let $\o\in{\R_{>0}}^n$ be a vector whose coordinates are $\Q$-linearly independent. The set
\begin{equation*}\begin{split}\mathcal{S}^\K_\o =\left\{  \vphantom{\frac{1}{kp^l}}\xi \text{ series } \mid  \exists k\in \N^*, \gamma\in\Z^n,  \sigma  \text{ a }\right.& \leq_\o\text{-positive rational cone, } \\
 \Supp  (\xi )\subset (\gamma +\sigma )\cap &\left.\frac{1}{k}\GG  \text{ and } \Supp(\xi) \text{ is }\leq_\o\text{-well-ordered}\right\}, \end{split}\end{equation*}
where 
$$\GG=
\bigcup_{\ell\in\N}\frac{1}{p^\ell}\Z^n ,$$
 is an algebraically closed field.
\end{thm}

We give here a positive characteristic version of $\mathcal S_\preceq^\K$:

\begin{defi}\label{def_field}
We fix an order $\preceq\in\Ord_n$ and a field $\K$ of characteristic $p> 0$. We set
\begin{equation*}\begin{split}\mathcal{S}^\K_{\preceq} :=\left\{  \vphantom{\frac{1}{kp^l}}\xi\text{ series }\mid  \exists k\in \mathbb{N}^{\ast},  \gamma\in\Z^n,  \sigma \text{ a } \preceq\text{-positive} \right.&\text{ rational cone containing } {\R_{\geq 0}}^n, \\
\text{such that }\Supp  (\xi )\subset (\gamma +\sigma )\cap \frac{1}{k}\GG,  \text{ and } &\left.  \vphantom{\frac{1}{kp^l}} 
 \Supp(\xi) \text{ is }\preceq \text{-well-ordered } \vphantom{\frac{1}{kp^l}} \right\}. \end{split}\end{equation*}
\end{defi}

Then the following result, extending Theorem \ref{saavedra} is the positive characteristic analogue of Theorem \ref{thm_alg_clos}:

\begin{thm}\label{thm_alg_clos_p}
Let $\preceq\in \Ord_n$ and $\K$ be an algebraically closed field of positive characteristic. Then the set $\mathcal S^\K_\preceq$ is an algebraically closed field containing $\K(\!(x)\!)$.
\end{thm}

In order to prove this theorem we will use the notion of field-family introduced by Rayner:

\begin{defi}\label{ff}\cite{ra}
A family $\mathcal{F}$ of subsets of an ordered abelian group $(G,\preceq)$ is said to be a field-family with respect to $G$ if we have the following.
\begin{enumerate}
\item  Every element of $\mathcal{F}$ is a well-ordered subset of $G$.
\item The elements of the members of $\mathcal{F}$ generate $G$ as an abelian group.
\item $\forall (A,B)\in \mathcal{F}^2, A \cup B \in \mathcal{F}$.
\item $\forall A \in \mathcal{F}$ and $B\subset A, B\in \mathcal{F}$.
\item $\forall (A,\gamma) \in \mathcal{F} \times G$, $\gamma +A \in\mathcal{F}$.
\item $\forall A \in \mathcal{F}$, if $A$ is $\preceq$-positive,  the semigroup generated by $A$  belongs to $\mathcal{F}$.
\end{enumerate}
\end{defi}

\begin{thm}\cite[Theorem 2]{ra}\label{rayner}
If $\mathcal F$ is a field-family with respect to $G$ then the set 
$$\left\{\sum_{g\in G}a_gx^g\mid \{g\mid a_g\neq0\}\in \mathcal F\right\}$$ is a Henselian valued field.
\end{thm}

For $\preceq\in\Ord_n$ we set 
\begin{equation*}\begin{split}\mathcal{F}_{\preceq} :=\left\{  \vphantom{\frac{1}{kp^l}}A\subset \Q^n\mid  \exists k\in \mathbb{N}^{\ast},  \gamma\in\Z^n,  \sigma \text{ a } \preceq\text{-positive rational cone containing }\right.& {\R_{\geq 0}}^n, \\
A\subset (\gamma +\sigma )\cap \frac{1}{k}\GG, \text{ and }   A\text{ is }\preceq\text{-well-ordered } &\left.  \vphantom{\frac{1}{kp^l}} \right\}. \end{split}\end{equation*}

\begin{prop}\label{p_ff} 
The set $\mathcal{F}_{\preceq}$ is a field-family with respect to $(\mathbb{Q}^n,\preceq)$.
\end{prop}

\begin{proof}
It is straightforward to verify that $\mathcal F_\preceq$ satisfies the  items (1), (2), (4) and (5) of Definition \ref{ff}. For (3), if $A$, $B\in \mathcal F_\preceq$, we have
$$A\subset (\gamma_A +\sigma_A )\cap \frac{1}{k_A}\GG,\ \  B\subset (\gamma_B +\sigma_B )\cap \frac{1}{k_B}\GG$$
 for some $\g_A$, $\g_B\in\Z^n$, $\s_A$ and $\s_B$ $\preceq$-positive rational cones containing ${\R_{\geq 0}}^n$ and $k_A$, $k_B\in \N^\ast$. We can replace $k_A$ and $k_B$ by their least common multiple and assume that $k_A=k_B$. We can also replace $\s_A$ and $\s_B$ by the cone $\s$ gerenated by $\s_A$ and $\s_B$. Since $\s_A$ and $\s_B$ are $\preceq$-positive and rational, $\s$ is also $\preceq$-positive and rational. Finally we may assume that $\g_A=\g_B$ by Lemma \ref{inter_cones}. Moreover $A$ and $B$ are $\preceq$-well-ordered, thus $A\cup B$ is $\preceq$-well-ordered. This shows that (3) is satisfied.\\
Therefore we only prove (6) here. The proof is done by induction on $n$. In fact  we will prove by induction on $n$, the following claim:\\
\textbf{Claim:} \emph{For $A\subset (\g+\s)$, where  $\g\in\Z^n$, $\s$ is a a strongly convex rational cone,  and $A$ is $\preceq$-positive and $\preceq$-well-ordered, there exists a $\preceq$-positive rational cone $\s'\supset\s$ such that 
$A\subset \s'$.}\\
This claim, along with the following theorem, proves the proposition:
   \begin{thm}\label{neumann}  \cite[Theorem 3.4, p. 206]{N}
  Let $A$ be a well-ordered subset of an ordered group $(G,\preceq)$. If  $A$ is $\preceq$-positive,  the semigroup generated by $A$ is well-ordered.
  \end{thm}
 
  Let us consider a set $A$ as in the claim.\\
  If $n=1$, there is only two orders on $\Q$. Both cases are symmetric, thus we may assume that $\preceq$ is the usual order $\leq$ on $\Q$ and $\s=\R_{\geq 0}$. Therefore we may assume that $\g=0$ as $A\subset \Q_{\geq 0}$. In this case $\mathcal U_\s=\{\leq\}$. Since $A$ is $\leq$-positive and $\leq$-well-ordered, $\lg A\rg\subset\Q_{\geq 0}$ is also $\leq$-well-ordered by Theorem \ref{neumann}.
   This settles the case $n=1$.\\
 So from now on, assume that $n>1$ and that the result is satisfied for $n-1$.\\
We know that there exist nonzero vectors $(u_1,\dots,u_s)\in (\mathbb{R}^n)^s$ and $(q_1,\dots,q_r) \in (\mathbb{Q}^n)^r$ such that $\preceq=\leq_{(u_1,\dots,u_s)}$ and $\sigma=\langle q_1,\dots,q_r\rangle$.\\
Assume first that $\gamma \succeq \underline0$. Then $A\subset \sigma '=\langle \gamma,q_1,\dots,q_r\rangle$ and $\sigma '$ is a $\preceq$-positive rational cone. Hence  $A$  is included in $\sigma ' \cap \frac{1}{k}\GG$, and the claim is proved. \\
\\
Now assume that $\gamma \prec \underline 0$. By replacing $\s$ by the cone generated by $\s$ and $-\g$, we may assume that $\underline0\in \g+\s$. We define $a:=\min (A\setminus\{\underline0\})$ and we set 
$$H:=\{u\in \mathbb{R}^n \text{ such that } u\cdot u_1=a\cdot u_1 \},$$
$$H^{+}:=\{u\in \mathbb{R}^n \text{ such that } u\cdot u_1 \geq a\cdot u_1\}.$$
By assumption,  $a \succ \underline0$.  Hence $a\cdot u_1 \geq 0$ because $\preceq=\leq_{(u_1,\dots,u_s)}$. \\
\textbf{Case 1:}  If $a\cdot u_1 >0$  we define $\s'$ to be the closure of the cone spanned by $H\cap(\g+\s)$. A vector $u\in\s'$ is not in the cone spanned by $H\cap(\g+\s)$ if and only if $\R_{\geq 0}u$ is the limit of halflines of the form $\R_{\geq 0}v_n$ with $v_n\in H\cap(\g+\s)$ and $(v_n)_n$ is not bounded. Therefore, we may assume that $v_n=\g+\la_nv'$ where $\la_n>0$ and $v'\in\s$ is orthogonal to $u_1$. Therefore $\s'$   is generated by the vertices of $H\cap(\g+\s)$ and the generators $\s\cap\langle u_1\rangle^\perp$, in particular $\s'$ is a rational cone.\\
Moreover $\s'$ is a $\preceq$-positive cone, because $u\succeq 0$ for every $u\in H\cap(\g+\s)$, and $(\gamma + \sigma) \cap H^{+} \subset \sigma '$. Finally it is clear that $\s'$ is strongly convex: if $u$, $-u\in \s'$, since $\s'$ is $\preceq$-positive, $u\in \s'\cap\langle u_1\rangle^\perp=\s\cap\langle u_1\rangle^\perp$; but $\s$ is strongly convex, thus $u=0$.\\
Now, if $a'\in A$, we have $a'\cdot u_1\geq a\cdot u_1$, thus there is $1\geq \la>0$ such that $\la a'\in H$. But we have
$$\la a'=\g+\la(a'-\g)-\g\in \g+\s$$
because $-\g\in\s$ and $a'-\g\in\s$. Therefore, $\la a'\in \s'$ and $A\subset \s'$. Thus the claim is proved in this case.\\
 \textbf{Case 2:} Assume that $a\cdot u_1=0$. We denote  the set $A\cap H\cap\Q^n$ by $B$, and we set $a_1 := \min (A \setminus B)$. Since $A \subset \{\delta \in \Q^n \mid \delta \succeq \underline0\}$ and $a_1 \notin H$, we have $a_1\cdot u_1 >0$. By Case 1, there exists a strongly convex rational $\preceq$-positive cone $\sigma_1$ containing $\s$ such that $A \setminus B \subset  \sigma_1$. \\
We have that $H\cap\Q^n$ is a $\Q$-vector space of dimension $d<n$. We set $V:=(H\cap\Q^n)\otimes_\Q\R$. We set $\s_2:=\s\cap V$ and we denote by $\preceq_V$ the restriction of $\preceq$ to $V$. Then $\s_2$ is a strongly convex rational  $\preceq_V$-positive cone.  If $\s_2$ is not full dimensional, we replace $\s_2$ by a strongly convex rational  $\preceq_V$-positive cone that is full dimensional. Therefore, by Lemma \ref{inter_cones}, we have that $B\subset \g_2+\s_2$ for some $\g_2\in V$. Therefore, by the inductive hypothesis, there is a strongly convex rational cone $\preceq_V$-positive $\s_3$ such that $\s_2\subset \s_3$ and $ B\subset \s_3$.\\
Now we set $\s':=\s_1+\s_3$. This cone is rational and $\preceq$-positive, thus it is strongly convex. Moreover it contains $A$, therefore the claim is proved.
\end{proof}

\begin{proof}[Proof of Theorem \ref{thm_alg_clos_p}]
By Proposition \ref{p_ff} and Theorem \ref{rayner}, the  set $\mathcal S^\K_\preceq$ is a Henselian valued field.\\
Assume that   $\mathcal S^\K_\preceq$ is not algebraically closed. Then, by \cite[Lemma 4]{ra} there exists $a\in \mathcal S^\K_\preceq$ such that $T^p-T-a$ is irreducible in $\mathcal S^\K_\preceq[T]$. Let us write
$$a=a^++a^-$$ 
where $\Supp(a^-)\subset \{b\in \Q^n\mid b\prec \underline 0\}$ and $\Supp(a^+)\subset \{b\in\Q^n\mid b\succeq \underline 0\}$. Because the map $b\lgm b^p$ is an additive map, if $\xi^+$ is a root of $T^p-T-a^+$ and $\xi^-$ is root of $T^p-T-a^-$, then $\xi^++\xi^-$ is a root of $T^p-T-a$. We will prove that $T^p-T-a^+$ and $T^p-T-a^-$ admit a root in $\mathcal S^\K_\preceq$ contradicting the fact that $T^p-T-a$ is irreducible.\\
\\
Since $\mathcal S_\preceq^\K$ is a Henselian valued field, 
$$\mathfrak O:=\left\{\xi\in\mathcal S_\preceq^\K\mid\forall b\in\Supp(\xi), b\succeq \underline 0\right\}$$
is a Henselian local ring with maximal ideal 
$$\mathfrak m:=\left\{\xi\in\mathcal S_\preceq^\K\mid\forall b\in\Supp(\xi), b\succ \underline 0\right\}.$$
The polynomial $T^p-T-a^+\in\mathfrak O[T]$ has a root modulo $\mathfrak m$ since $\K$ is algebraically closed (here $\mathfrak O/\mathfrak m=\K$). Moreover the derivative of this polynomial is -1. Thus this polynomial satisfies Hensel's Lemma and admits a root $\xi^+$ in $\mathcal S^\K_\preceq$.\\
\\
In order to prove that $T^p-T-a^-$ has a root in $\mathcal S^\K_\preceq$, we follow the proofs of \cite[Theorem 3]{ra}, and \cite[Theorem 5.3]{S}. We write $a^-=\sum_{q\in\Q^n} a^-_q x^q$ and we define
$$\xi^-:=\sum_{q\in\Q^n}\left(\sum_{i=1}^\infty \left(a^-_{p^iq}\right)^{\frac{1}{p^i}}\right)x^q.$$
We can verify that $\xi^-$ is well defined: for a given $q\in\Supp(a^-)$,  the sequence $(p^iq)_i$ is strongly decreasing for the order $\preceq$ since $q\prec \underline 0$. Therefore $a^-_{p^iq}=0$ for $i$ large enough because $\Supp(a^-)$ is $\preceq$-well-ordered. Hence the sum  $\sum_{i=1}^\infty \left(a^-_{p^iq}\right)^{\frac{1}{p^i}}$ is in fact a finite sum.\\
Then we remark that
$$\Supp(\xi^-)\subset \bigcup_{i\in\N^\ast}\frac{1}{p^i}\Supp(a^-),$$
thus $\Supp(\xi^-)$ is $\preceq$-well-ordered by \cite[Lemma 5.2]{S}. \\
Finally we claim that $\Supp(\xi^-)$ is contained in the translation of a rational $\preceq$-positive cone. In order to prove this we assume that $\Supp(a^- )\subset \g+\s$ where $\g\in\Z^n$ and $\s$ is a rational $\preceq$-positive cone, and we denote by $\g_1$, \ldots, $\g_s\in\Z^n$ some generators of $\s$.\\
First we assume that $\g\succeq \underline 0$. Let $\a\in\Supp(\xi^-)$, $\a=\frac{1}{p^i}\a'$ with $\a'\in\Supp(a^-)$. We have
$$\frac{1}{p^i}\a'+\g=\frac{1}{p^i}(\a'-\g)+\left(\frac{1}{p^i}+1\right)\g\in \s_1$$
where $\s_1$ is the cone generated by the $\g_i$ and $\g$. Thus $\a\in-\g+\s_1$, which proves the claim because $\s_1$ is rational and $\preceq$-positive.\\
Now assume that $\g\prec \underline 0$ and consider $\a\in\Supp(\xi^-)$ written $\a=\frac{1}{p^i}\a'$ as before. Then $\frac{1}{p^i}(\a'-\g)\in\s$ because $\a'\in\g+\s$. Thus
$$\frac{1}{p^i}\a'-\g=\frac{1}{p^i}(\a'-\g)+\left(1-\frac{1}{p^i}\right)(-\g)\in\s_2$$
where $\s_2$ is the cone generated by the $\g_i$ and $-\g$. Thus $\a\in\g+\s_2$, which proves the claim because $\s_2$ is rational and $\preceq$-positive.\\
\\
 Moreover an easy computation shows that $\xi^-$ is a root of $T^p-T-a^-$. This proves the theorem.
\end{proof}

\subsection{Examples}
We do not know if Theorem \ref{main_thm} remains valid for elements of $\mathcal S_\preceq^\K$ when $\K$ is positive characteristic field, but all the other results proved before in characteristic zero are non longer true in positive characteristic, as shown by the following examples:

\begin{ex}\label{ex4}
Let $\K$ be a field of characteristic $p>0$. Set $f=\displaystyle\sum_{k=1}^\infty t^{1-\frac{1}{p^k}}$. The series $f$ is algebraic over $\K(t)$ because $f^p-t^{p-1}f-t^{p-1}=0$. Thus $g:=\displaystyle\sum_{k=1}^\infty\left(\frac{x}{y}\right)^{1-\frac{1}{p^k}}$ is algebraic over $\K(x,y)$. We set $\displaystyle\xi=\sum_{k=1}^\infty (xg)^k$. Because $\xi=\frac{xg}{1-xg}$, $\xi$ is rational over the field extension of $\K(x,y)$ by $g$. Hence $\xi$ is algebraic over $\K(x,y)$. \\
We see that all the monomials of $(xg)^k$ are of the form $x^{k-l}y^l$ for $l\in\Q_{\geq 0}$. Therefore the support of $\xi$ is included in the cone $\s$ generated by $(2,-1)$ and $(0,1)$ (see Figure \ref{fig_ex4}). Moreover the support of $(xg)^k$ contains a sequence of points converging to $(2k,-k)$. But $(2k,-k)$ does not belong to the support of $\xi$ since $(1,-1)$ does not belong to the support of $g$. Hence $\t(\xi)=\s^\vee$ is generated by $(1,0)$ and $(1,2)$.\\
But the conclusions of Theorem \ref{main_thm3} and \ref{main_thm2}  do not hold in this case: there is  no hyperplane $H_\la=\{(x,y)\in\R^2\mid x+2y=\la\}$ containing infinitely many elements of $\Supp(\xi)$ such that $H_\la^-:=\{(x,y)\in\R^2\mid x+2y<\la\}$ contains only finitely many elements of $\Supp(\xi)$.

\tdplotsetmaincoords{60}{120} 
\begin{center}\begin{figure}[H]\fbox{\begin{tikzpicture} [scale=0.8, axis/.style={->,thin}, 
vector/.style={-stealth,red,very thin}, 
vector guide/.style={dashed,red,thick}]

\coordinate (O) at (0,0);

\pgfmathsetmacro{\ax}{0.8}
\pgfmathsetmacro{\ay}{0.8}

\coordinate (P) at (\ax,\ay);

\draw[axis] (0,0) -- (5.2,0) node[anchor= west]{$x$};
\draw[axis] (0,0,0) -- (0,2.2) node[anchor=north west]{$y$};

 \draw (1,0)circle (0.pt)  node{} --  (5/3,-2/3) circle (0.pt) node {} -- (17/9,-8/9)  circle (0.pt) node {} -- (26/27+1,-26/27)  circle (0.pt) node {}-- (80/81+1,-80/81)  circle (0.pt) node {} -- (2,-1) circle (1.5pt) node{} ;
 
 \draw (2,0) circle (0.pt)  node{} --  (1+1+5/3,-1-2/3) circle (0.pt) node {} -- (2+17/9,-1-8/9)  circle (0.pt) node {} -- (2+26/27+1,-1-26/27)  circle (0.pt) node {}-- (2+80/81+1,-1-80/81)  circle (0.pt) node {} -- (4,-2) circle (1.5pt) node{}  ;
 
 \draw (3,0) circle (0.pt)  node{} --  (4+5/3,-2-2/3) circle (0.pt) node {} -- (4+17/9,-2-8/9)  circle (0.pt) node {} -- (5+26/27,-2-26/27)  circle (0.pt) node {}-- (4+80/81+1,-2-80/81)  circle (0.pt) node {} -- (6,-3) circle (1.5pt) node{}  ;
 
 \fill [color=gray, opacity=0.1]
(0,0) -- (0,2) -- (6,2) -- (6,-3) -- cycle;

\end{tikzpicture}}\caption{Example \ref{ex4}}\label{fig_ex4}\end{figure}\end{center}
Here $\t'_0(\xi)=\emptyset$. This shows that Lemma \ref{compar} is not valid in general for generalized series with exponents in $\Q^n$ that are algebraic over $\K(\!(x)\!)$, for a positive characteristic field $\K$.
\end{ex}


\begin{ex}\label{ex_pos}
We set $a=\sum_{i=1}^\infty x^iy^{-1}\in \mathbb F_2(\!(x,y)\!)$ and $P(T)=T^2+T+a$. For $i\in\N^\ast$ we also denote $P_i(T)=T^2+T+x^iy^{-1}$. We consider an order $\preceq\in\Ord_2$.\\
The roots of $P_i$ in $\SS_\preceq^{\mathbb F_2}$ are
$$\left\{\begin{array}{c}\xi_i^{(1)} \text{ and } \xi_i^{(1)}+1, \text{ with }\xi_i^{(1)}=\sum_{k=1}^\infty x^{i2^{-k}}y^{-2^{-k}} \text{ when } (-i,1)\succ \underline 0 \\
\xi_i^{(2)} \text{ and } \xi_i^{(2)}+1, \text{ with }\xi_i^{(2)}=-\sum_{k=1}^\infty x^{i2^{k}}y^{-2^{k}} \text{ when } (i,-1)\succ \underline 0
\end{array}\right.$$
Let $i_0=\sup\{i\in\N^\ast \mid (-i,1)\succ \underline 0\}\in\N^\ast\cup\{\infty\}$.
Therefore the roots of $P$ in $\SS_\preceq^{\mathbb F_2}$ are $\xi_\preceq$ and $\xi_\preceq+1$ where 
$$\xi_\preceq=\sum_{i=1}^{i_0}\xi_i^{(1)}+\sum_{i_0+1}^\infty \xi_i^{(2)}.$$

\tdplotsetmaincoords{60}{120} 
\begin{center}\begin{figure}[H]\fbox{\begin{tikzpicture} [scale=1.2, axis/.style={->,thin}, 
vector/.style={-stealth,red,very thin}, 
vector guide/.style={dashed,red,thick}]

\coordinate (O) at (0,0);


\pgfmathsetmacro{\ax}{0.8}
\pgfmathsetmacro{\ay}{0.8}

\coordinate (P) at (\ax,\ay);

\draw (1,-1) circle (1.2pt);
\draw (2,-1) circle (1.2pt);
\draw (3,-1) circle (1.2pt);
\draw (4,-1) circle (1.2pt);
\draw (5,-1) circle (1.2pt);
\draw (6,-1) circle (1.2pt);

\draw[axis] (-0.5,0) -- (11.5,0) node[anchor= west]{$x$};
\draw[axis] (0,-3) -- (0,1) node[anchor=north west]{$y$};

\draw[fill] (1/2,-1/2) circle (0.8pt); 
\draw[fill]  (1/4,-1/4) circle (0.8pt);
\draw[fill]  (1/8,-1/8)  circle (0.8pt);
\draw[fill] (1/16,-1/16) circle (0.8pt);

\draw[fill] (2/2,-1/2) circle (0.8pt); 
\draw[fill]  (2/4,-1/4) circle (0.8pt);
\draw[fill]  (2/8,-1/8)  circle (0.8pt);
\draw[fill] (2/16,-1/16) circle (0.8pt);

\draw[fill] (3/2,-1/2) circle (0.8pt); 
\draw[fill]  (3/4,-1/4) circle (0.8pt);
\draw[fill]  (3/8,-1/8)  circle (0.8pt);
\draw[fill] (3/16,-1/16) circle (0.8pt);

\draw[fill] (8,-2) circle (0.8pt); 
\draw[fill]  (10,-2) circle (0.8pt);
\draw[fill]  (12,-2)  circle (0.8pt);

   \draw[gray, dashed, ultra thin] (0,0)   --   (1,-1) ;
  \draw[gray, dashed, ultra thin] (0,0)   --   (2,-1) ;
 \draw[gray, dashed, ultra thin] (0,0)   --   (3,-1) ;
 \draw[gray, dashed, ultra thin] (4,-1)   --   (10,-2.5) ;
  \draw[gray, dashed, ultra thin] (5,-1)   --   (12,-2.4) ;
   \draw[gray, dashed, ultra thin] (6,-1)   --   (12,-2) ;

\end{tikzpicture}}\caption{Example \ref{ex_pos}}\label{fig_last_ex}
\end{figure}\end{center}

We can replace $P(T)$ by $\wdt P(T):=T^2+yT+\sum_{i=1}^\infty x^i\in\K[[ x]][T]$ and remark that $\wdt P(yT)=y^2P(T)$, thus the roots of $P_1(T)$ are obtained from those of $P(T)$ by multiplication by $y$.
This proves that Proposition \ref{prop_key} is no longer valid in positive characteristic.
 \end{ex}


\begin{ex}\label{ex_pos_3}
Let $\K=\mathbb F_2$ be the field with two elements. The series 
$$a(x,y)=x\sum_{k=1}^\infty\left(\frac{x}{y}\right)^{1-2^{-k}}$$ is algebraic over $\mathbb F_2(x,y)$. Thus the roots of $T^2+T+a$ are algebraic over $\mathbb F_2(x,y)$. One of these roots is
$$\xi=-\sum_{k=1}^\infty\sum_{\ell=0}^\infty\left(\frac{x^{2-2^{-k}}}{y^{1-2^{-k}}}\right)^{2^\ell}\in\mathcal S_\preceq^{\mathbb F_2}$$
where $\preceq\in \Ord_2$ is such that $(1,-1)\succ \underline 0$. The
 support of $\xi$ is given on Figure \ref{fig_last_ex2} below. Thus, $\t(\xi)^\vee$ is the cone generated by $(2,-1)$ and $(0,1)$. Here $\t_0(\xi)$ is not  open  since $(1,1)\in\t_0(\xi)$: here $\t_0(\xi)$ is equal to the cone generated by $(1,0)$ and $(1,1)$ minus the origin. Thus Proposition \ref{prop_open} is no longer valid in positive characteristic. We remark that $\t'_0(\xi)=\emptyset$.\\
 On the following picture, the small circles indicate the terms of the support of $a(x,y)$, while the bullets indicate the terms of the support of $\xi$:

\tdplotsetmaincoords{60}{120} 
\begin{center}\begin{figure}[H]\fbox{\begin{tikzpicture} [scale=1.2, axis/.style={->,thin}, 
vector/.style={-stealth,red,very thin}, 
vector guide/.style={dashed,red,thick}]

\coordinate (O) at (0,0);


\pgfmathsetmacro{\ax}{0.8}
\pgfmathsetmacro{\ay}{0.8}

\coordinate (P) at (\ax,\ay);

\draw (1.5,-0.5) circle (1.2pt);
\draw  (1.75,-0.75) circle (1.2pt);
\draw (1+7/8,-7/8) circle (1.2pt);
\draw (1+15/16,-15/16) circle (1.2pt);

\draw[axis] (-0.5,0) -- (7,0) node[anchor= west]{$x$};
\draw[axis] (0,-3) -- (0,1) node[anchor=north west]{$y$};
\draw[fill] (3,-1) circle (0.8pt); 
\draw[fill]  (6,-2) circle (0.8pt);

%

\draw[fill]  (3.5,-1.5)  circle (0.8pt);
\draw[fill] (7,-3) circle (0.8pt);
\draw[fill] (2+7/4,-7/4) circle (0.8pt); 
\draw[fill]  (4+7/2,-7/2) circle (0.8pt);
\draw[fill]  (2+15/8,-15/8)  circle (0.8pt);
\draw[fill] (4+15/4,-15/4) circle (0.8pt);
%
%
%
%
%
   \draw[gray, dashed, ultra thin] (0,0)   --   (4*1.5,-2) ;
  \draw[gray, dashed, ultra thin] (0,0)   --   (4+3,-3) ;
 \draw[gray, dashed, ultra thin] (0,0)   --   (4+7/2,-7/2) ;
 \draw[gray, dashed, ultra thin] (0,-0)   --   (4+15/4,-15/4) ;
  \draw[gray, dashed, ultra thin] (0,0)   --   (4+31/8,-31/8) ;
%
\end{tikzpicture}}\caption{Example \ref{ex_pos_3}}\label{fig_last_ex2}
\end{figure}\end{center}

\end{ex}




\end{document}